\documentclass[10pt]{article}
\usepackage[english,activeacute]{babel}
\usepackage[T1]{fontenc}

\usepackage{graphicx,amsmath,amssymb,amsthm,xcolor,wasysym,bbm}
\colorlet{labelkey}{blue}
 \usepackage{bbm}
\newcommand{\sech}{\operatorname{sech}}

\newcommand{\UU}{U}

\renewcommand{\S}{\mathbb{S}}

\newcommand{\R}{\mathbb{R}}
\newcommand{\C}{\mathbb{C}}

\newcommand{\N}{\mathbb{N}}

\newcommand{\boR}{\mathcal{R}}
\newcommand{\boL}{\mathcal{L}}
\newcommand{\boT}{\mathcal{T}}

\newcommand{\ptl}{{\partial}}
 \providecommand{\abs}[1]{|#1 |}
\newcommand{\wh}{\widehat }
\newcommand{\E}{\mathcal E}

\newcommand{\MM}{\mathcal M }
\newcommand{\vp}{\varphi }

 \newcommand{\loc}{\mathrm{loc}}
 \providecommand{\norm}[1]{\|#1 \|}

\renewcommand{\div}{\operatorname{div}}

\newcommand{\curl}{\operatorname{curl}}

\newcommand{\ve}{\varepsilon }
\newcommand{\vr}{\varrho }

\newcommand{\intRR}{\int_{\R^N}}
\newcommand{\intRx}[1]{\int_{\R^2}#1\,dx}
\newtheorem{teo}{Theorem}[section]
\newtheorem{definition}[teo]{Definition}
\newtheorem{prop}[teo]{Proposition}
\newtheorem{lema}[teo]{Lemma}
\newtheorem{lemma}[teo]{Lemma}

\newtheorem{cor}[teo]{Corollary}

\newtheorem{step}{Step}
  \renewcommand{\S}{\mathbb{S}}

\newcommand{\intR}{\int_{\R^2}}

\theoremstyle{definition}
\newtheorem{remark}[teo]{Remark}
\newcommand{\grad}{\nabla}
 \newcommand{\osc}{\operatornamewithlimits{osc}}
\newtheorem*{merci}{Acknowledgments}

\title{Minimal energy for the traveling waves of the Landau--Lifshitz equation}

\author{Andr\'e de Laire\\
{Laboratoire Paul Painlev\'e, Universit\'e de Lille 1}\\
{59655 Villeneuve d'Ascq Cedex, France}\\
{\tt andre.de-laire@math.univ-lille1.fr}\\
}
\date{}

\begin{document}
 \maketitle
\begin{abstract} 
We consider nontrivial finite energy traveling waves for the Landau--Lifshitz equation with easy-plane anisotropy.
Our main result is the existence of a minimal energy for these traveling waves, in dimensions two, three and four.
The proof relies on a priori estimates related with the theory of harmonic maps and
the connection of the Landau--Lifshitz equation with the kernels appearing in the
Gross--Pitaevskii equation.
\medskip
 \\
%
%

\end{abstract}

\section{Introduction}
 \setcounter{equation}{0}
   \numberwithin{equation}{section}
In this work we consider the Landau--Lifshitz equation
\begin{equation}\label{LL}
\ptl_t m+m\times( \Delta m+\lambda m_3 e_3)=0, \qquad m(t,x)\in \S^2, \ t\in \R, \ x\in \R^N,
\end{equation}
where $e_3=(0,0,1)$, $\lambda\in \R$ and $m=(m_1,m_2,m_3)$.
This equation  was originally introduced by L.~Landau and E.~Lifshitz in \cite{landaulifshitz}
 to describe the dynamics of magnetization in ferromagnetic materials. Here  the parameter $\lambda$ takes
into account the anisotropy of such material. More precisely, the value $\lambda=0$ corresponds to the isotropic
case, meanwhile  $\lambda>0$ and $\lambda<0$ correspond to materials with
an easy-axis and an easy-plane  anisotropy, respectively (see \cite{kosevich,hubert}).

The isotropic case $\lambda=0$ recovers the Sch\"odinger map equation,
which has been intensively studied due to its applications in several areas
of physics and mathematics (see \cite{guo,shatah}). For $\lambda>0$, the existence of solitary waves periodic in time have been established
in \cite{gustafson-LL,piette}. Moreover, Pu and Guo \cite{pu} showed that $\lambda\neq 0$ is a necessary condition to the existence of
these types of solutions.

In this paper we are interested in the case of easy-plane anisotropy $\lambda<0$. By a scaling
argument we can suppose from now on that $\lambda=-1$. Then the energy of \eqref{LL} is given by
\begin{equation*}
E(m)=\intRR e(m)\,dx\equiv\frac12\intRR \left( \abs{\grad m}^2+m_3^2 \right)dx,
\end{equation*}
that it is formally conserved due to the Hamiltonian structure of \eqref{LL}.
If $m$ is smooth, by differentiating twice the condition $\abs{m(t,x)}^2=1$ we obtain $m\cdot \Delta m=-\abs{\grad m}^2$,
so that  taking cross product of $m$ and \eqref{LL}, we can recast \eqref{LL} as
\begin{equation}\label{LL2}
m\times \ptl_t m=\Delta m+\abs{\grad m}^2m -(m_3 e_3-m_3^2 m).
\end{equation}
Using formal developments and numerical simulations,
Papanicolaou and Spathis \cite{papanicolaou} found in dimensions $N\in\{2,3\}$ nonconstant finite energy
 traveling waves of \eqref{LL2}, propagating with speed $c\in(0,1)$ along the $x_1$-axis, i.e.
of the form
\begin{equation*}
m_c(x,t)=u(x_1-ct,x_2,\dots,x_N).
\end{equation*}
By substituting $m_c$ in \eqref{LL2}, the profile $u$ satisfies
\begin{equation}\label{TW-LL}
-\Delta u=\abs{\grad u}^2 u +u_3^2 u-u_3 e_3+ c u\times \ptl_1 u.
 \tag{{TW}$_c$}
\end{equation}

Notice that if  $u$ satisfies \eqref{TW-LL} with speed $c$, so does $-u$ with speed $-c$, therefore we can assume that $c\geq 0$.
Also, we see that any constant in  $\S^1\times\{0\}$ satisfies  \eqref{TW-LL}, so that we refer to them as the trivial solutions.
Since we are interested in finite energy solutions, the natural energy space to work in is
\begin{equation*}
  \E(\R^N)=\{v\in L^1_{\loc}(\R^N;\R^3) : \grad v\in L^2(\R^N), \ v_3\in L^2(\R^N), \ \abs{v}=1 \textup{  a.e. on } \R^N \}. \qquad
\end{equation*}

\begin{subsection}{The minimal energy}\label{intro-1}
Our main theorem is in the same spirit as the result proved by the author for the Gross--Pitaevskii
equation in \cite{delaire-cras} (see also \cite{bethuel}). Precisely, we show the existence of
a minimal value for the energy for the nontrivial traveling waves.
\begin{teo}\label{teo-non-N}
Let $N\in\{2,3,4\}$. There exists a universal constant $\mu>0$ such that if $u\in \E(\R^N)$ is a nontrivial solution of \eqref{TW-LL}
with $c\in (0,1]$, satisfying in addition that $u$ is uniformly continuous if $N\in\{3,4\}$, then
\begin{equation}\label{min-energy-N}
E(u)\geq \mu.
\end{equation}
\end{teo}

As noticed in \cite{hang} in dimension two, there is no smooth static solution of \eqref{TW-LL}, i.e. with speed $c=0$.
More generally, we obtain the following result for static waves.
\begin{prop}\label{prop-static-L}
Let $N\geq 2$. Assume that   $u\in \E(\R^N)$ is a solution of \eqref{TW-LL} with $c=0$. Suppose also that $u$ is uniformly continuous
if $N\geq 3$. Then $u$ is a trivial solution.
\end{prop}

Theorem~\ref{teo-non-N} shows that there are no small energy traveling wave solutions in dimensions two, three and four
(assuming that they are uniformly continuous in dimensions three and four).
This opens the door to have a scattering theory for equation \eqref{LL} with $\lambda=-1$, similarly to the theory developed
for the Gross--Pitaevskii equation by Gustafson, Nakanishi and Tsai~\cite{nakanishi,nakanishi2}.

The one-dimensional case is different. If $N=1$, \eqref{TW-LL} is completely integrable and we can compute the solutions in $\E(\R)$ explicitly.
More precisely,
\begin{prop}\label{sol-N-1-intro}
Let $N=1$, $c\geq 0$ and $u\in \E(\R)$ be solution of \eqref{TW-LL}.
\begin{enumerate}
\item[(i)] If $c\geq 1$, then $u$ is a trivial solution.
\item[(ii)] If $0\leq c <1$ and $u$ is nontrivial, then, up to invariances, $u$ is given by
$$
u_1=c \sech(\sqrt{1-c^2}\,x),\ u_2=\tanh(\sqrt{1-c^2}\,x),\  u_3=\sqrt{1-c^2}\sech(\sqrt{1-c^2}\,x).
$$
Moreover, if $0<c<1$,
\begin{equation*}
E(u) =2\sqrt{1-c^2} \quad \text{ and }\quad E(p(u)) =2\sin(p(u)/2),
\end{equation*}
where $p(u)$ denotes the momentum of $u$.
\end{enumerate}
\end{prop}
We notice that equation \eqref{TW-LL} is invariant under translations and under the action of $\S^1$ by a rotation around the $e_3$-axis,
that is if $u=(u_1,u_2,u_3)$ is a solution of \eqref{TW-LL}, so is $$(u_1\cos(\varphi)-u_2\sin(\varphi),u_1\sin(\varphi)+u_2\cos(\varphi),u_3),$$
for any $\varphi\in\R$. Also, if $u=(u_1,u_2,u_3)(x)$ is a solution,
so is  $u=(u_1,u_2,-u_3)(-x)$. These are the invariances that we refer to in Proposition~\ref{sol-N-1-intro}.

We provide the proof of Proposition~\ref{sol-N-1-intro} in Section~\ref{sec-N-1}, as well as the precise definition of momentum.
The relation $E=2\sin(p/2)$ is showed in Figure~\ref{fig-E-p-N-1}. In particular we note that there are solutions of small energy, but
there is a maximum value for the energy and the momentum.
We also remark that Proposition~\ref{sol-N-1-intro} provides a solution with $c=0$, meanwhile
Proposition~\ref{prop-static-L} states that this is not possible in the case  $N\geq 2$.
\begin{figure}[ht]
\begin{center}
\scalebox{0.73}{\includegraphics{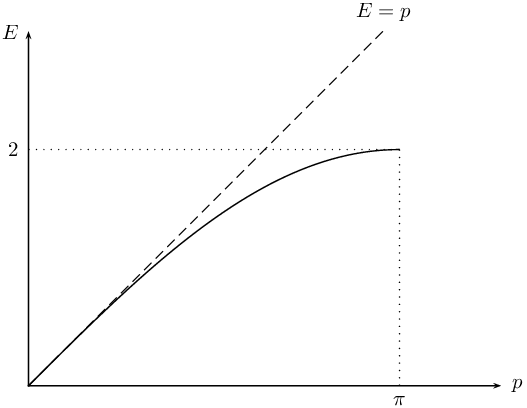}}
\end{center}
\caption{Curve of energy $E$ as a function of the momentum $p$ in the one-dimensional case.}
\label{fig-E-p-N-1}
\end{figure}
\end{subsection}

\begin{subsection}{From Gross--Pitaevskii to Landau--Lifshitz}
The results of this paper have been motivated by the numerical simulations in \cite{papanicolaou},
where the authors determine a branch of nontrivial solutions of  \eqref{TW-LL}, axisymmetric around the $x_1$-axis,
for any speed $c \in(0,1)$ in dimensions two and three. They also
conjecture that there is no nontrivial finite energy solution of \eqref{TW-LL} for $c\geq 1$.
For $c$ small, the existence of these traveling waves has been proved rigorously by Lin and Wei~\cite{wei}.
The branch of solutions is depicted in Figure~\ref{fig-E-p}.  We see that the curve has a nonzero minimum,
which represents the minimal energy in Theorem~\ref{teo-non-N}.
\begin{figure}[ht]
\begin{center}
\scalebox{0.7}{\includegraphics{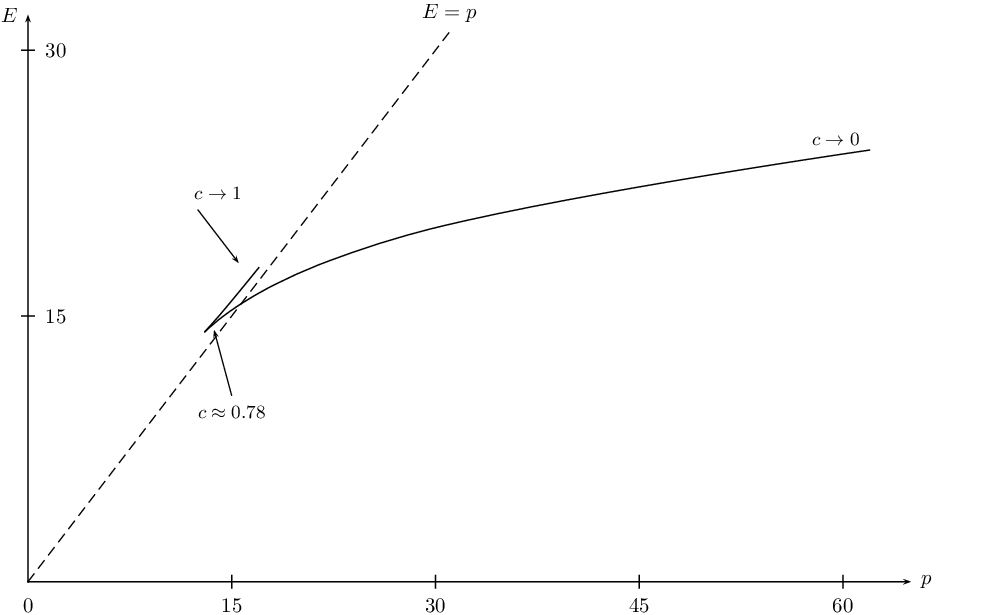}}
\end{center}
\caption{Curve of energy $E$ as a function of the momentum $p$ in the two-dimensional case.}
\label{fig-E-p}
\end{figure}

Some properties of the solutions found in \cite{papanicolaou} are very similar to those of  the traveling waves
for the Gross--Pitaevskii equation obtained numerically by Jones, Putterman and Roberts \cite{JPR1,JPR2} and
 studied rigourously in \cite{bethuel0,bethuel,maris}. In fact, if $u$ is a solution of \eqref{TW-LL}, the stereographic variable
\begin{equation*}
\psi=\frac{u_1+i u_2}{1+u_3},
\end{equation*}
satisfies
\begin{equation}\label{stereo}
\Delta \psi +\frac{1-\abs{\psi}^2}{1+\abs{\psi}^2}\psi -ic\ptl_{1}\psi=\frac{2\bar{\psi}}{1+\abs{\psi}^2}(\grad \psi )^2,
\end{equation}
that seems like a perturbed equation for the traveling waves for the Gross--Pitaevskii equation, namely
\begin{equation}\label{GP}
\Delta \Psi +(1-\abs{\Psi}^2)\Psi -ic\ptl_{1}\Psi=0.
\end{equation}

However, other properties of the solutions are very different.
For instance, the energy-momentum curve for the Gross--Pitaevskii equation tends to zero as the momentum goes to zero in the two-dimensional case,
but there exists a minimal energy if $N\geq 3$ (see \cite{bethuel,delaire-cras}).

From a mathematical point of view,
\eqref{stereo} is a quasilinear Schr\"odinger equation meanwhile \eqref{GP} is a semilinear Schr\"odinger equation. Therefore, it is not clear
how to relate both equations. One of the purposes of this paper is to clarify this connection, to show how to exploit the similarities between
\eqref{stereo} and \eqref{GP}, and how to deal with the extra difficulties of equation \eqref{TW-LL}. In particular, we will discuss
the regularity of the solutions of  \eqref{TW-LL}, some a priori bounds and their asymptotic behavior as $\abs{x}\to\infty$.
\end{subsection}

\begin{subsection}{Sketch of the proof of Theorem~\ref{teo-non-N}}\label{subsec}
The starting point of our analysis  is
that for any solution $u\in \E(\R^N)$ of \eqref{TW-LL}, there exists  $ R \equiv R(u)$ such that we have the lifting
\begin{equation}\label{lifting-LL}
\check u\equiv u_1+i u_2=\vr e^{i\theta}, \qquad\textup{ on }B(0,R)^c,
\end{equation}
where $\vr\equiv \sqrt{u_1^2+u_2^2}=\sqrt{1-u_3^2}$ and  $\vr,\theta \in \dot H^1(B(0,R)^c)$ (see Lemma~\ref{lifting-tilde-E}).
Let $\chi\in C^\infty(\R^N)$ be such that $\abs{\chi}\leq 1$, $\chi=0$ on $B(0,2R)$ and $\chi=1$ on $B(0,3R)^c$,
if $R>0$. In the case that $R=0$, we let $\chi=1$ on $\R^N$.
In this way,  we can assume that the function $\chi \theta$
and
\begin{equation}\label{G-LL}
G=(G_1,\dots,G_N)\equiv u_1\grad u_2-u_2\grad u_1-\grad(\chi \theta),
\end{equation}
is well-defined on $\R^N$.
For $u=(u_1,u_2,u_3)$, equation \eqref{TW-LL} reads
\begin{align}
 -\Delta u_1&=2e(u) u_1+c(u_2\ptl_1 u_3-u_3\ptl_1u_2),\label{u1}\\
 -\Delta u_2&=2e(u) u_2+c(u_3\ptl_1 u_1-u_1\ptl_1u_3),\label{u2}\\
-\Delta u_3&=2e(u) u_3-u_3+c(u_1\ptl_1 u_2-u_2\ptl_1 u_1). \label{u3}
\end{align}
Then, using \eqref{u1} and \eqref{u2},
\begin{equation}\label{eq-div-G}
\begin{split}
\div(G)&=u_1\Delta u_2-u_2\Delta u_1-\Delta (\chi \theta)\\
&=c( \ptl_1 u_3 -u_3 u\cdot \ptl_1 u)-\Delta(\chi \theta)\\
&=c\ptl_1 u_3-\Delta(\chi \theta),
\end{split}
\end{equation}
where we used the fact that $u\cdot \ptl_1 u=0$. By combining with \eqref{u3}, we obtain
\begin{equation}\label{bi-LL}
 \Delta^2 u_3-\Delta u_3+c^2\ptl^2_{11}u_3=-\Delta F+c\ptl_1(\div G), \qquad \textup{ on }\R^N.
 \end{equation}
At this point we remark that the differential operator
\begin{equation*}
\Delta^2 -\Delta +c^2\ptl^2_{11}
\end{equation*}
is  elliptic if and only if $c\leq 1$, which shows that $c=1$
is a critical value for the equation \eqref{TW-LL}.

Taking Fourier transform in \eqref{bi-LL}, we get
\begin{equation}\label{id-int-LL}
\left(\abs{\xi}^4+\abs{\xi}^2-c^2\xi_1^2\right) \wh u_3(\xi)=\abs{\xi}^2 \wh F(\xi)-c\sum_{j=1}^N \xi_1\xi_j \wh G_j(\xi),
\end{equation}
and hence
\begin{equation}\label{eq-fourier}
\wh u_3(\xi)=L_c(\xi)\left(\wh F(\xi)-c\sum_{j=1}^N c\frac{\xi_1\xi_j}{\abs{\xi}^2}\wh G_j(\xi)\right),
\end{equation}
where
\begin{equation*}
L_c(\xi)=\frac{\abs{\xi}^2}{\abs{\xi}^4+\abs{\xi}^2-c^2\xi_1^2}.
\end{equation*}
Equivalently, we can write \eqref{eq-fourier} as the convolution equation
\begin{equation}\label{conv-u3}
u_3=\mathcal L_c*F-c \sum_{j=1}^N\mathcal L_{c,j}* G_j,
\end{equation}
where $\wh{ \mathcal L}_c=L_c$ and
\begin{equation}\label{kernels}
 \wh {\mathcal L}_{c,j}=\frac{\xi_1\xi_j}{\abs{\xi}^4+\abs{\xi}^2-c^2\xi_1^2}.
\end{equation}
Similarly, from \eqref{eq-div-G} and \eqref{eq-fourier}, for $j\in\{1,\dots, N\}$,
\begin{equation}\label{conv-theta}
\ptl_j(\chi \theta)=c\, \boL_{c,j}*F-c^2\sum_{k=1}^N \boT_{c,j,k}*G_k-\sum_{k=1}^N \boR_{j,k}*G_k,
\end{equation}
where
\begin{align*}
   \wh \boT_{c,j,k}=\frac{\xi_1\xi_j\xi_k}{\abs{\xi}^2({\abs{\xi}^4+\abs{\xi}^2-c^2\xi_1^2)}} \qquad \textup{and}\qquad\wh \boR_{j,k}=\frac{\xi_j\xi_k}{\abs{\xi}^2},
\end{align*}
for all $j,k\in\{1,\dots,N\}$.

The kernels $\mathcal L_{c}$, $\mathcal L_{c,j}$, $\boR_{j,k}$ and $\boT_{c,j,k}$
are the same as those appearing in the Gross--Pitaevskii equation \eqref{GP}. As a consequence,
all the properties valid for the \eqref{GP} depending only in the
structure of these kernels, can be transfer to the Landau--Lifshitz equation.
For instance, the asymptotic behavior theory  developed in \cite{gravejat-decay,gravejat-first,bona-li,debouard-saut}
can be applied, after proving some  algebraic decay of the solutions. We provide a precise statement in Theorem~\ref{limite-infty} at the end of this introduction.

Roughly speaking, the principle used to find a minimal energy for the traveling waves of \eqref{GP} in \cite{bethuel,delaire-cras},
written in the context of the equation \eqref{TW-LL}, is that on
one hand a convolution equation such as \eqref{conv-u3} should imply that
\begin{equation}\label{E1}
\norm{u_3}_{L^p}\leq C(\norm{u}_{W^{k,q}})E(u)^\gamma, \tag{E1}
\end{equation}
for some $q,k\in \N$, and $\gamma>0$.
On the other hand, using \eqref{u1}, \eqref{u2}, \eqref{u3} and integrating by parts one should get an a priori
bound for the energy of the form:
\begin{equation}\label{E2}
E(u)\leq C(\norm{u}_{W^{\ell,r}}) \norm{u_3}_{L^p}^\delta, \tag{E2}
\end{equation}
for some $\ell,r\in \N$, and $\delta>0$. By putting together \eqref{E1} and \eqref{E2},
$$E(u)\leq C(\norm{u}_{W^{\ell,r}}) C(\norm{u}_{W^{k,p}})^{\delta}E(u)^{\gamma\delta}.$$
Notice that we can assume that $E(u)>0$ because $u$ is not constant.
If  $\gamma\delta>1$ and if
\begin{equation}\label{E3}
  C(\norm{u}_{W^{\ell,r}}) C(\norm{u}_{W^{k,p}})^{\delta}\leq M, \tag{E3}
\end{equation}
for some constant $M$ independent of $u$ and $c$,
we can conclude that
\begin{equation*}
\frac{1}{M^{1/(\gamma\delta-1)}}\leq E(u),
\end{equation*}
so that we have the existence of a minimal energy.

In conclusion, we have reduced the proof of Theorem~\ref{teo-non-N} to the proof of the estimates \eqref{E1}, \eqref{E2} and \eqref{E3},
for some $\gamma,\delta>0$ such that  $\gamma\delta>1$.

\subsubsection{Estimate \eqref{E3} and regularity of traveling waves}

Let us consider the quasilinear elliptic system,
\begin{equation*}
\Delta u=f(x,u,\grad u),\qquad \textup{ in } \Omega,
\end{equation*}
where $\Omega$ is a smooth domain and $f$ is a smooth function with quadratic growth
\begin{equation*}\abs{f(x,z,p)}\leq A+B \abs{p}^2.
 \end{equation*}
We notice that the square-gradient term prevents us from invoking the usual elliptic regularity estimates.
However, well-known regularity results imply that every \textit{continuous} solution in $H^1(\Omega)$ belongs to
$H^{2,2}_{\loc}(\Omega)\cap C_{\loc}^{0,\alpha}(\Omega)$  (see \cite{giaquinta,lady,borchers,jost}), but in general we do not have nice a priori estimates such as
in the $L^p$-regularity theory because the $H^{2,2}_{\loc}(\Omega)$-norm depends on the modulus of continuity of the $u$.
To exemplify this point, let us consider the harmonic map equation
\begin{equation}\label{eq-contra-ex-L}
-\Delta v=\abs{\grad v}^2 v, \qquad \textup{ in }\Omega, \ v\in \S^2.
\end{equation}
Let $\Omega=\R^N$, $N\geq 2$, and assume that there exists $C_1,C_2,\alpha>0$ such that
\begin{equation}\label{contra-ex2}
\norm{\grad v}_{L^\infty(\R^N)}\leq C_1+C_2\norm{\grad v}_{L^2(\R^N)}^\alpha.
\end{equation}
Note that $\norm{\grad v}_{L^2(\R^N)}^2$ is the energy associated to \eqref{eq-contra-ex-L}.
Since the function  $v_\lambda(x)=v(\lambda x)$, $\lambda>0$, also solves \eqref{eq-contra-ex-L},
we conclude that $v_\lambda$ satisfies \eqref{contra-ex2}, but this implies
\begin{equation*}
\norm{\grad v}_{L^\infty(\R^N)}\leq \frac{1}{\lambda}\left(C_1+\frac{C_2}{\lambda^{\alpha(N-2)/2}}\norm{\grad v}_{L^2(\R^N)}^\alpha\right).
\end{equation*}
Then, letting $\lambda\to \infty$, we deduce that $v$ is constant. Therefore an estimate such as \eqref{contra-ex2} does not hold for \eqref{eq-contra-ex-L}
and probably neither for \eqref{TW-LL}. This a big difference with the semilinear equation \eqref{GP}. Indeed, if $\Psi$ is a solution of \eqref{GP}, then (see \cite{farina,bethuel})
$$\norm{\Psi}_{C^k(\R^N)}\leq C(c,k,N).$$

In dimension $N=2$, H\'elein~\cite{helein0,helein} proved that any finite energy solution of \eqref{eq-contra-ex-L}
is continuous and therefore smooth. In dimension $N\geq 3$, this result if false. In fact, if $N\geq 3$,
$v(x)= x/\abs{x}$ is a discontinuous finite energy solution and Rivi\`ere~\cite{riviere} proved
that \eqref{eq-contra-ex-L} has almost everywhere discontinuous solutions with finite energy.

For these reasons, we need to treat differently the cases $N=2$ and $N\geq 3$. In any case, we establish in Section~\ref{sec-reg}
a bound of $\norm{\grad u}_{L^\infty(\R^N)}$ in terms of the energy, provided that the energy is small enough.

\begin{prop}\label{teo-regularidad}
Let $c\geq 0$ and $u\in\E(\R^2)$ be a solution of \eqref{TW-LL}. Then $u\in C^\infty(\R^2)$,
$u_3\in L^p(\R^2)$ for all $p\in [2,\infty]$ and
$\grad u \in {W^{k,p}(\R^2)}$
for all $k\in \N$ and $p\in [2,\infty]$.
Moreover, there exist constants $\ve_0>0$ and $K>0$, independent of $u$ and $c$, such that
\begin{align}
 \norm{u_3}_{L^\infty(\R^2)}&\leq K(1+c)E(u)^{1/2}, \label{u_3-small0}\\
\norm{\grad u}_{L^\infty(\R^2)}&\leq K(1+c)E(u)^{1/4},\label{grad-E-small0}
\end{align}
provided that $E(u)\leq \ve_0$.
\end{prop}

Denoting by $UC(\R^N)$ the set of uniformly continuous functions, in the higher dimensional case, we have
\begin{prop}\label{lema-reg-N}
Let $N\geq 3$, $c\geq 0$ and $u\in\E(\R^N)\cap UC(\R^N)$ be a  solution of \eqref{TW-LL}.
Then $u\in C^\infty(\R^N)$ and
$\grad u\in W^{k,p}(\R^N)$ for all $k\in \N$ and $p\in[2,\infty]$. Moreover,
if $N\in \{3,4\}$ and $c\in[0,1]$, there exist $\ve_0,K,\alpha>0$, independent of $u$ and $c$, such that
\begin{equation}\label{est-N-3}
\norm{u_3}_{L^\infty(\R^N)}\leq KE(u)^{\alpha},\end{equation}
 \begin{equation}\label{est-N-3-grad}
\norm{\grad u}_{L^\infty(\R^N)}\leq  KE(u)^{\alpha},\end{equation}
provided that $E(u)\leq \ve_0$.
\end{prop}

As we will show, Propositions \ref{teo-regularidad} and \ref{lema-reg-N} will be enough to get the universal constant $M$
in estimate \eqref{E3}. The proof of Proposition~\ref{lema-reg-N} is the only point of the paper where the condition $N\leq 4$
is used. It is straightforward to verify that if  the estimates \eqref{est-N-3} and \eqref{est-N-3-grad} are satisfy for some
dimension $N$, then Theorem~\ref{teo-non-N} holds for this $N$.

\subsubsection{Estimates \eqref{E1} and \eqref{E2}}
In order to prove Theorem~\ref{teo-non-N}, we can assume that $E(u)$ is small. Then, by \eqref{u_3-small0} and \eqref{est-N-3},
we only need to prove \eqref{E1} and \eqref{E2} for traveling waves such that $\norm{u_3}\leq 1/2$.

In Section~\ref{sec-conv}, we will prove that estimate \eqref{E1} holds with $p=4$, $\gamma=1$ if $N=2$,  $p=2$, $\gamma=\frac{2N+3}{2(N-1)}$ if $N\geq 3$,
and $k=1$, $q=\infty$ in both cases. The main element in the proof is the study of the Fourier multiplier $\mathcal L_c$ done by the author in \cite{delaire-cras}
if $N\geq 3$ and by Chiron and Maris \cite{maris-chiron} if $N=2$.

Sections~\ref{sec-poh} and \ref{sec-vortexless} are devoted
to establish some Pohozaev identities and a priori bounds that allow us to obtain estimate \eqref{E2}. More precisely, under the condition
$\norm{u_3}\leq 1/2$, we show that
\begin{equation*}
E(u)\leq K\norm{u_3}_{L^p}^\delta,
\end{equation*}
with $p=4$, $\delta=4$ if $N=2$, and $p=2$, $\delta=2$ if $N\geq 3$.
\end{subsection}

\begin{subsection}{Asymptotic behavior  at infinity}\label{sec-state}
As remarked before, the arguments given by Gravejat in \cite{gravejat-decay,gravejat-first}
apply to \eqref{conv-u3} and \eqref{conv-theta}, since they rely mainly on the structure of the kernels.
This allows to establish the  precise limit at infinity of the finite energy solutions of \eqref{TW-LL}.
\begin{teo}\label{limite-infty}
Let $N\geq 2$ and $c\in (0,1)$. Assume that $u\in \E(\R^N)$ is a solution of \eqref{TW-LL}.
Suppose further that  $u\in UC(\R^N)$ if $N\geq 3$.
Then there exist a  constant $\lambda_\infty\in \C$ of modulus one
and two functions $\check u_\infty, u_{3,\infty}\in C(\S^{N-1};\R)$ such that
\begin{align}
\abs{x}^{N-1}(\check u(x)-\lambda_\infty)-i \lambda_\infty \check u_\infty\left(\frac x{\abs{x}}\right)\to0,\label{lim-u-check}\\
\abs{x}^N u_3(x)- u_{3,\infty}\left(\frac x{\abs{x}}\right)\to0,\label{lim-u3}
\end{align}
uniformly as $\abs{x}\to\infty$. Moreover, assuming without loss of generality  that $\lambda_\infty=1$,
we have
\begin{align}
\check u_\infty(\sigma)&=
\frac{\alpha \sigma_1}{(1-c^2+c^2\sigma_1^2)^{\frac N2}}+ \sum_{j=2}^N\frac{\beta_j \sigma_j}{(1-c^2+c^2\sigma_1^2)^{\frac N2}},\label{id-inf1}\\
  u_{3,\infty}(\sigma)&=\alpha c\left(\frac{1}{(1-c^2+c^2\sigma_1^2)^\frac{N}{2}}-\frac{N\sigma_1^2}{(1-c^2+c^2\sigma_1^2)^\frac{N+2}{2}}\right)
- \sum_{j=2}^N \beta_j \frac{Nc\sigma_1\sigma_j}{(1-c^2+c^2\sigma_1^2)^{\frac{N+2}2}},\label{id-inf2}
 \end{align}
where $\sigma=(\sigma_1,\dots,\sigma_N)\in \S^{N-1}$,
\begin{align*}
\alpha&=\frac{\Gamma\left(\frac{N}{2}\right)}{2\pi^\frac{N}2}(1-c^2)^{\frac{N-3}{2}}\left(2c\intRR e(u)u_3\,dx-(1-c^2)\intRR G_1(x)\,dx \right)
\end{align*}
and
\begin{align*}
\beta_j&=-\frac{\Gamma\left(\frac{N}{2}\right)}{2\pi^\frac{N}2}(1-c^2)^{\frac{N-1}{2}}\intRR G_j(x)\,dx.
\end{align*}
\end{teo}

In particular, since the solutions found in \cite{wei}  are uniformly continuous,
Theorem~\ref{limite-infty} applies to those solutions.
For the sake of completeness we sketch the proof of Theorem~\ref{limite-infty} in Section~\ref{sec-decay}.
\begin{remark}
The analogous constants $\alpha$ and $\beta_j$ found in \cite{gravejat-first}
for the asymptotic behavior at infinity of the nontrivial finite energy traveling waves to the Gross--Pitaevskii equation \eqref{GP}
can be written in terms the energy and (vectorial) momentum. Later, Wei and Yao \cite{wei-yao} have proved that 
all the coordinates of the momentum of these solutions are zero, except the first one. 
As a consequence the traveling waves are asymptotically axisymmetric (see \cite{wei-yao} for details).
It is not clear if the arguments in \cite{wei-yao} can be generalized in the context of the Landau--Lifshitz equation. However, we think that the
same conclusion holds in our case, that is $\beta_j=0$, for all $2\leq j\leq N$, which would imply the asymptotically axisymmetry of the solutions of \eqref{TW-LL}.
\end{remark}

\end{subsection}

\paragraph{Notations.} We use the standard notations ``$\cdot$'' and ``$\times$'' for the inner and cross product, respectively.

For $y\in\R^N$ and $r\geq 0$, $B(y,r)$ or $B_r(y)$ denote the open ball of center $y$ and radius $r$ (which is empty for $r=0$).
In the case that there is no confusion, we simply put $B_r$.


Given $x=(x_1,x_2)$, $f:\R^2\to \R^2$, $f=(f_1,f_2)$,  we set $x^\bot=(-x_2,x_1)$,  and $\curl(f)=\ptl_1 f_2-\ptl_2 f_1 $.
We also use the skew gradient $\grad^{\perp}=(-\ptl_2,\ptl_1)$.

For a function  $g:\R^\ell\to \R^3$, $g=(g_1,g_2,g_3)$,
 we define $\check g$ as the complex-valued function $\check g=g_1+ig_2$.
We identify $\grad g$ with the matrix in $\R^{\ell,3}$ whose columns are $\grad g_1$, $\grad g_2$ and $\grad g_3$.

Let $A=[A_1 \, | \, A_2 \, | \, A_3]$ and $\tilde A=[\tilde A_1 \, | \, \tilde A_2 \, | \, \tilde A_3]$ be two matrices in $\R^{\ell,3}$, then
\begin{equation*}A: \tilde A=\sum_{j=1}^3 A_j\cdot \tilde A_j,\end{equation*}
and for any vector $b\in\R^3$, $A.b\in \R^3$. denotes the standard matrix-vector product.

Either $\mathcal F(f)$ or $\wh f$ stands for the  Fourier transform of $f$, namely
\begin{equation*}\mathcal F(f)(\xi)=\wh f(\xi)=\intRR f(x)e^{-ix\cdot\xi}\,dx.\end{equation*}
We also adopt the standard notation $K(\cdot,\cdot, \dots )$ to represent a generic constant
that depends only on each of its arguments.
\section{Estimates for $\norm{u_3}_{L^\infty(\R^N)}$ and $\norm{\grad u}_{L^\infty(\R^N)}$}\label{sec-reg}
In this section we use some of the elements developed to the study of the harmonic map equation.
In particular, the next lemma is a consequence of the Wente  lemma \cite{wente,brezis-coron0,topping} and H\'elein's trick \cite{helein0,helein}.
\begin{lemma}\label{lemma-osc}
Let $\Omega\subset \R^2$ be a smooth bounded domain and $g\in L^2(\Omega)$.
 Assume that  $u\in H^1(\Omega,\S^2)$ satisfies
\begin{equation}\label{eq-g}
-\Delta u=\abs{\grad u}^2 u +g,\qquad \textup{ in }\Omega.
 \end{equation}
Let $r>0$ and $x\in \Omega$ such that $B(x,r)\subseteq \Omega$. Then
 for any  $i\in\{1,2,3\}$ we have
\begin{equation} \label{osc-complete}
\begin{split}
 \osc_{B(x,r/2)}u_i\leq  & K\Big(\min\left\{\norm{\grad u}_{L^2(B(x,r))},  \norm{u_i}_{L^\infty(\ptl B_r)}\right\}+ \norm{\grad u}^2_{L^2(B(x,r))}\\
&+r \norm{g}_{L^2(B(x,r))} \big(1+\norm{\grad u}_{L^2(B(x,r))}\big)\Big),
 \end{split}
\end{equation}
for some universal constant $K>0$. In particular $u\in C(\Omega)$. Moreover, if the trace of $u$
on $\ptl \Omega$ belongs to $C(\ptl \Omega)$, then  $u\in C(\bar \Omega)$ and
\begin{equation} \label{osc-complete2}
\norm{u_i}_{L^\infty(\Omega)}\leq \norm{u_i}_{L^\infty(\ptl \Omega)}+  K(\Omega)(\norm{\grad u}^2_{L^2(\Omega)}
+\norm{g}_{L^2(\Omega)} (1+\norm{\grad u}_{L^2(\Omega)})), \quad
\end{equation}
for some constant $K(\Omega)$ depending only on $\Omega$.
\end{lemma}

\begin{proof}
As for the standard harmonic maps, we recast \eqref{eq-g} as
\begin{equation} \label{eq-1}
-\Delta u_i=\sum_{j=1}^3 v_{i,j}\cdot \grad u_j+g_i, \qquad \textup{ in }\Omega, \quad i=1,2,3,\end{equation}
 where $v_{i,j}=u_i\grad u_j-u_j\grad u_i$. Then
\begin{equation}\label{div}
\div(v_{i,j})=u_jg_i-u_ig_j\quad  \textup{ and } \quad \norm{\div(v_{i,j})}_{L^2(\Omega)}\leq 2\norm{g}_{L^2(\Omega)}.
\end{equation}
Let us consider  $h_{i,j} \in H^2(\Omega)$ the solution of
 \begin{equation}
\label{eq-h}
\begin{cases}
 \Delta h_{i,j}=\div(v_{i,j}), & \text{in }  \Omega,\\
  h_{i,j}=0, & \text{on } \ptl \Omega.
\end{cases}
\end{equation}
Thus
\begin{equation}\label{cota-h}
\norm{\grad h_{i,j}}_{L^2(\Omega)}\leq \norm{v_{i,j}}_{L^2(\Omega)}\leq 2 \norm{\grad u}_{L^2(\Omega)}.\end{equation}
Since $\div(v_{i,j}-\grad h_{i,j})=0$, with  $v_{i,j}-\grad h_{i,j}\in L^2(\Omega)$, there exists
$w_{i,j}\in H^1(\Omega)$ (see e.g. \cite[Thm 2.9]{raviart}) such that
\begin{equation}\label{eq-p} v_{i,j}=\grad h_{i,j}+\grad^\bot w_{i,j}, \qquad \textup{ in } \Omega.\end{equation}

Now we decompose $u$ as $u_i=\phi_i+\varphi_i+\psi_i$, where $\phi_i,\varphi_i,\psi_i$ are the solutions
of the equations
\begin{equation}\label{harmonica}
\begin{cases}
   -\Delta \phi_i=0, &\text{in }  \UU,\\
    \phi_i=u_i,\quad & \text{on } \ptl \UU,
\end{cases}
\end{equation}
\begin{equation}\label{tilde-q-LL}
\begin{cases}
 -\Delta \varphi_i=\grad h_i :  \grad u+g_i, &\text{in }  \UU,\\
\varphi_i=0, &\text{on } \ptl \UU,
\end{cases}
\end{equation}
\begin{equation}\label{q}
 \begin{cases}
 -\Delta \psi_i=\grad^\bot w_i : \grad u, &\text{in } \UU,\\
\psi_i=0, &\text{on } \ptl \UU,
\end{cases}
\end{equation}
where $\grad h_i=[\grad h_{i,1} \, | \, \grad h_{i,2} \, | \, \grad h_{i,3}]$, $\grad^{\perp} w_i=[\grad^{\perp} w_{i,1} \, | \, \grad^{\perp} w_{i,2} \, | \, \grad^{\perp} w_{i,3}]$, and
$\UU$ is an open smooth domain such that $B_r\subseteq \UU\subseteq\Omega$.
We now prove \eqref{osc-complete} for $r=1$, supposing that $B_1\subseteq \UU$,
since then  \eqref{osc-complete} follows from a scaling argument.
First, invoking Theorem~\ref{harmonic} we  have that
\begin{equation}\label{dem-77-1}
\osc_{B_{1/2}} \phi_i\leq K\min\left\{\norm{\grad \phi_i}_{L^2(B_{2/3})}, \norm{\phi_i}_{L^2(B_{2/3})}\right\}.
\end{equation}
Also, some standard computations and the maximum principle yield
\begin{equation}\label{dem-77-2}
 \norm{\grad \phi_i}_{L^2(\UU)}\leq \norm{\grad u_i}_{L^2(\UU)} \quad\textup{and}
\quad \norm{\phi_i}_{L^\infty(\UU)}\leq \norm{u_i}_{L^\infty(\ptl \UU)}.
\end{equation}
Thus from \eqref{dem-77-1} and \eqref{dem-77-2} we conclude that
 \begin{equation}\label{osc1}
\osc_{B_{1/2}} \phi_i\leq K\min\left\{\norm{\grad u_i}_{L^2(\UU)}, \norm{u_i}_{L^\infty(\ptl \UU)}\right\}.
\end{equation}
For $\varphi_i$,  Theorem~\ref{stam-osc} gives
\begin{equation}\label{osc1-1}
\osc_{B_{1}} { \varphi_i }\leq K(\norm{\grad h\cdot \grad u}_{L^{3/2}(B_1)}+\norm{g}_{L^2(B_1)}).
\end{equation}
To estimate the first term in the r.h.s. of \eqref{osc1-1}, we use
the H\"older inequality
\begin{equation}\label{osc1-1-22}
\norm{\grad h_i : \grad u}_{L^{3/2}(B_1)}\leq  \norm{\grad h}_{L^{6}(B_1)}\norm{\grad u}_{L^{2}(B_1)},
\end{equation}
and the Sobolev embedding theorem
\begin{equation}\label{osc1-1-23}
\norm{\grad h}_{L^{6}(B_1)}\leq K(\norm{\grad h}_{L^{2}(B_1)}+\norm{D^2 h}_{L^{2}(B_1)}).
\end{equation}
By using \eqref{div}, \eqref{osc1-1}, \eqref{osc1-1-22},  \eqref{osc1-1-23}  and $L^2$-regularity estimates for \eqref{eq-h}, we are led to
\begin{equation}\label{osc2} \osc_{B_{1}}{\varphi_i }\leq K\norm{g}_{L^{2}(B_1)} (1+ \norm{\grad u}_{L^{2}(B_1)}).
\end{equation}
Similarly, since $W^{2,p}(\UU)\hookrightarrow C(\bar{\UU})$, for all $p>1$, we also have
 \begin{equation}\label{norm-varphi}
\norm{ \varphi_i }_{C(\bar \UU)}\leq K(\UU) \norm{g}_{L^{2}(\UU)} (1+ \norm{\grad u}_{L^{2}(\UU))}).
 \end{equation}
To estimate $\psi_i$ we invoke the Wente estimate (see \cite{topping}, \cite{helein}), so that
\begin{equation} \label{osc3}
\norm{\psi_i}_{C(\bar \UU)}+\osc_{\UU}\psi_i\leq
K\norm{\grad w}_{L^{2}(\UU)}\norm{\grad u}_{L^{2}(\UU)}\leq K \norm{\grad u}_{L^{2}(\UU)}^2,
\end{equation}
 where we have used  \eqref{cota-h} and  \eqref{eq-p} for the last inequality.

Therefore, taking $\UU=B_1$ and putting together  \eqref{osc1}, \eqref{osc2} and \eqref{osc3}, we conclude \eqref{osc-complete}
 with $r=1$.

If the trace of $u$
on $\ptl \Omega$ belongs to $C(\ptl \Omega)$,  we take $\Omega=\UU$ and then from \eqref{harmonica}
we deduce that $\phi_i\in C^2(\Omega)\cap C(\bar \Omega)$. Since $\varphi_i,\psi_i\in C(\bar \Omega)$,
we conclude that $u_i\in C(\bar \Omega)$ and \eqref{osc-complete2} follows from \eqref{dem-77-2},
\eqref{norm-varphi} and \eqref{osc3}.
\end{proof}


\begin{lemma}\label{lemma-grad}
Let $y\in \R^2$,  $r>0$ and $B_r\equiv B(y,r)$. Assume that  $u\in H^1(B_r,\S^2)$  satisfies
\begin{equation}\label{eq-g2}
-\Delta u=\abs{\grad u}^2 u +f(x,u(x),\grad u(x)),\qquad   \textup{ in }B_r,
 \end{equation}
where $f$ is a continuous function such that  $\abs{f(x,z, p)}\leq C_1+C_2\abs{p}$,
for some constants $C_1,C_2\geq 0$, for a.e. $x\in B_r$, $z\in \R^3$, $p\in\R^{3\times 3}$.
Suppose that
\begin{equation}\label{cond-osc}
A\equiv A(u,r)\equiv \frac{\osc_{B_r}u(1+r^2 (C_1+C^2_2)) }{1-3\osc_{B_r}u} \leq \frac{1}{32}.
\end{equation}
Then
\begin{equation}\label{cota-D2u}
 \norm{D^2 u}_{L^2(B_{r/2})}+\norm{\grad u}_{L^4(B_{r/2})}^2\leq  Kr^{-1} \left(\norm{\grad u}_{L^2(B_r)}+\norm{g}_{L^2(B_r)}\right),
\end{equation}
where $g(x)=f(x,u(x),\grad u(x))$. Assume further that $f(x,z,p)=\tilde f(x)+R_f(x,z,p)$, for some continuous functions  $\tilde f$,
$R_f$, such that  $\abs{R_f(x,z,p)}\leq C_3\abs{p}$,
for some constant $C_3\geq 0$, for a.e. $x\in B_r$, $z\in \R^3$, $p\in\R^{3\times 3}$. Then,
\begin{equation}\label{grad-LL}
\begin{split}
  \norm{\grad u}_{L^\infty(B_{r/4})}\leq& Kr^{-1}\norm{\grad u}_{L^2(B_{r})} +
Kr^{-2/3}
 \Big(
  \norm{\grad u}_{L^2(B_r)}^2(r^{-2}+r^{-4/3}) +\norm{g}_{L^2(B_r)}^2
\\
 &+ \norm{\tilde f}_{L^3(B_{r})}
 +C_3r^{-1/3}\norm{\grad u}_{L^2(B_r)}^{1/3}\big(\norm{\grad u}_{L^2(B_r)}^{1/3}+\norm{g}_{L^2(B_r)}^{1/3}\big)
\Big),
\end{split}
  \end{equation}
where $K$ is  some universal constant.
\end{lemma}
\begin{proof}
As mentioned before, Lemma~\ref{lemma-osc} and the quadratic growth of the r.h.s. of \eqref{eq-g2}
imply that $u\in H_{\loc}^{2,2}(\Omega)$. In fact, this could be seen by repeating the following arguments
with finite differences instead of weak derivatives.
 As standard in the analysis of this type of equations, we
let $\rho\in(0,r)$ and $\chi\in C^\infty_0(B_r)$, with $\chi(x)=1$ if $\abs{x}\leq \rho$,
\begin{equation}\label{grad-chi}
\abs{\chi}\leq 1\quad\textup{and}\quad\abs{\grad \chi }\leq K/(r-\rho),\qquad\textup{ on }B_r.
\end{equation}
 Then setting $\eta=\chi \abs{\grad u}$, taking inner product
in \eqref{eq-g} with $(u-u(x_0))\eta^2$ and integrating by parts we obtain
\begin{equation}\label{dem-grad}
  \int_{B_r} \abs{\grad u}^2\eta^2+ 2 \int_{B_r} (\grad{u}.(u-u(x_0)))\cdot(\eta\grad \eta)
=\int_{B_r} \abs{\grad u}^2u\cdot (u-u(x_0))\eta^2+
\int_{B_r} \eta^2 g\cdot (u-u(x_0)).
\end{equation}
Then, using  the elementary inequality $2ab\leq a^2+b^2$,
\begin{equation*}
\begin{split}
 \abs{\eta^2 g\cdot (u-u(x_0))}\leq &\eta^2(C_1+C_2\abs{\grad u})\osc_{B_r}(u)\leq
C_1\eta^2 \osc_{B_r}(u)+\\
&\eta^2\abs{\grad u}^2\osc_{B_r}(u)+\frac14 C_2^2\eta^2\osc_{B_r}(u).
\end{split}
\end{equation*}
In a similar fashion, we estimate the remaining  terms in \eqref{dem-grad}. Then, using the Poincar\'e inequality
\begin{equation*}\norm{\eta}_{L^2(B_r)}\leq \frac{r}{j_0} \norm{\grad \eta}_{L^2(B_r)},\end{equation*}
where $j_0\approx 2.4048$ is the first zero of the Bessel function, and that $\abs{u}=1$, we conclude that
\begin{equation*}
\int_{B_r} \abs{\grad u}^2\eta^2\leq
 \frac{\osc_{B_r}u(1+r^2 (C_1+C^2_2))}{1-3\osc_{B_r}u} \int_{B_r} \abs{\grad \eta}^2,
\end{equation*}
where we bounded $1/j_0$ and $1/(4j_0)$ by 1 to simplify the estimate.
Thus,
\begin{equation}\label{grad-1}
\int_{B_r} \abs{\grad u}^4\chi^2\leq
 A \int_{B_r}(\abs{D^2u}^2\chi^2 +\abs{\grad u}^2\abs{\grad \chi}^2)).
\end{equation}

On the other hand,  taking inner product
in \eqref{eq-g2} with $\ptl_k(\chi^2\ptl_k u)$, integrating by parts and summing over $k=1,2$, we have
\begin{equation}
  -\int_{B_r}\chi^2\abs{D^2 u}^2-2
\sum_{\begin{array}{l}
{ i\in\{1,2\} \atop j,k\in\{1,2,3\}}
\end{array}
}
\int_{B_r}\ptl_{jk}{u_i}\chi \ptl_j \chi \ptl_{k}u_i=
\sum_{i\in\{1,2\}} \int_{B_r} (\abs{\grad u}^2u_i+g_i)(2\chi \grad \chi\grad  u_i+\chi^2 \Delta u_{i} ).
\end{equation}
Using again the inequalities $2ab\leq \ve a^2+b^2/{\ve}$ and $ab\leq \ve a^2+b^2/{4\ve}$, we are led to
\begin{equation}
  \int_{B_r}\chi^2\abs{D^2 u}^2\leq \frac1{1-3\ve} \int_{B_r} \left(
(2+\ve^{-1})\abs{\grad u}^2\abs{\grad \chi}^2  +
(1+4\ve^{-1})\abs{\grad u}^4\chi^2 +
(1+4\ve^{-1})\chi^2 \abs{g^2})\right).
\end{equation}
Then, minimizing with respect to $\ve$, it follows that
\begin{equation}\label{grad-2}
\int_{B_r}\chi^2\abs{D^2 u}^2\leq 16 \int_{B_r} ( \abs{\grad u}^2\abs{\grad \chi}^2  +\abs{\grad u}^4\chi^2 +\chi^2 \abs{g^2}).
\end{equation}
By combining \eqref{grad-chi}, \eqref{grad-1} and \eqref{grad-2}, we infer that
\begin{align}\label{estimation-jost1}
\int_{B_\rho} \abs{\grad u}^4&\leq \frac{KA}{1-16A}\left(\frac{1}{(r-\rho)^2}\int_{B_r}\abs{\grad u}^2+\int_{B_r}\abs{g}^2\right),\\
  \int_{B_\rho} \abs{D^2 u}^2&\leq \frac{K}{1-16A}\left(\frac{1+A}{(r-\rho)^2}\int_{B_r}\abs{\grad u}^2+\int_{B_r}\abs{g}^2\right).\label{estimation-jost2}
 \end{align}
Taking $\rho=r/2$ and using that $A\leq 1/32$,  \eqref{cota-D2u} follows.

Now we decompose $u$ as
$u_i=\phi_i+\psi_i$, where
\begin{equation}\label{decom1}
\begin{cases}
 -\Delta \phi_i=0,  &\text{in }  B_{r/2},\\
\phi_i=u_i, &\text{on } \ptl B_{r/2},
\end{cases}
\end{equation}
\begin{equation}\label{decom2}
 \begin{cases}
 -\Delta \psi_i=\abs{\grad u}^2u_i+\tilde f_i+(R_f(x,u,\grad u))_i, &\text{in }  B_{r/2},\\
\psi_i=0, &\text{on } \ptl B_{r/2},
\end{cases}
\end{equation}
Since $\phi_i$ is a harmonic function,
\begin{equation*}
\norm{\grad \phi_i}_{L^\infty(B_{r/4})}\leq Kr^{-1}\norm{\grad \phi_i}_{L^2(B_{r/2})},
\end{equation*}
so that using also \eqref{dem-77-2}, we obtain the estimate
\begin{equation}\label{estimation-1}
\norm{\grad \phi_i}_{L^\infty(B_{r/4})}\leq Kr^{-1}\norm{\grad u_i}_{L^2(B_{r/2})}.
\end{equation}
For $\psi_i$, we recall that using the $L^p$-regularity theory for the Laplacian
and a scaling argument, the solution $v\in H_0^1(B_R)$ of the equation $-\Delta v=h$ satisfies
\begin{equation*}
\norm{\grad v}_{L^\infty(B_R)}\leq K(p)R^{1-2/p}\norm{h}_{L^p(B_R)},\quad \textup{for all }p>2.
\end{equation*}
Applying this estimate with $p=3$ to \eqref{decom2}, we get
\begin{equation}\label{cota-gradinfty}
  \norm{\grad \psi_i}_{L^\infty(B_{r/2})}\leq Cr^{-2/3}\left(\norm{\grad u}_{L^6(B_{r/2})}^2+\norm{\tilde f}_{L^3(B_{r/2})}+C_3\norm{\grad u}_{L^3(B_{r/2})}\right).
\end{equation}
Also, by the Sobolev embedding theorem, we have
 \begin{equation}\label{dem-inter}
 \norm{\grad u}_{L^6(B_{r/2})}\leq K\left( \norm{D^2 u}_{L^2(B_{r/2})}+r^{-2/3}\norm{\grad u}_{L^2(B_{r/2})}\right).
 \end{equation}
Therefore, by putting together \eqref{cota-D2u}, \eqref{estimation-1}, \eqref{cota-gradinfty}, \eqref{dem-inter}
and the interpolation inequality
\begin{equation*}
\norm{\grad u}_{L^3(B_{r/2})}\leq \norm{\grad u}_{L^2(B_{r/2})}^{1/3}\norm{\grad u}_{L^4(B_{r/2})}^{2/3},
\end{equation*}
we deduce \eqref{grad-LL}.
\end{proof}

Now we turn back to equation \eqref{TW-LL}.  By setting
\begin{equation*} E_{x,r}(u)=\int_{B(x,r)}e(u),\end{equation*}
 we obtain the following result.

\begin{cor}\label{cor-est}
There exist $\ve_0>0$ and a positive constant $K(\ve_0)$,
such that for any $c\geq0$ and  any $u\in\E(\R^2)$  solution of \eqref{TW-LL}
 satisfying
\begin{equation*}
E_{x,r}(u)\leq \ve_0,
\end{equation*}
for some $x\in \R^2$ and $r\in (0,1]$,
we have
 \begin{align}
\osc_{B(x,r/2)} {u}&\leq K(\ve_0)(1+c)E_{x,r}(u)^{1/2},\label{osc-E}\\
\norm{\grad u}_{L^\infty(B(x,r/4))}&\leq K(\ve_0)(1+c)E_{x,r}(u)^{1/4}r^{-2/3}.\label{grad-E}
\end{align}
In particular,  if $E(u)\leq \ve_0$, then \eqref{u_3-small0} and \eqref{grad-E-small0} hold.
\end{cor}
\begin{proof}
Estimates \eqref{osc-E} and \eqref{grad-E}  follow from Lemmas \ref{lemma-osc} and \ref{lemma-grad}.
Then, taking $r=1$, we conclude that  \eqref{grad-E-small0} holds. Now we turn to \eqref{u_3-small0}.
For any $y\in \R^2$ we have
\begin{equation}\label{dem-eq-os1}
2 E(u)\geq \int_{B(y,1/2)} u_3^2\geq \frac{\pi}{4}\min_{B(y,1/2)}\abs{u_3}^2.
\end{equation}
On the other hand, by Lemma \ref{lemma-osc},
\begin{equation}\label{dem-eq-os2}
   \max_{B(y,1/2)}\abs{u_3}\leq \osc_{B(y,1/2)} u_3+\min_{B(y,1/2)} \abs{u_3}\leq K(1+c)E(u)^{1/2}+\min_{B(y,1/2)} \abs{u_3}.
\end{equation}
By combining \eqref{dem-eq-os1} and \eqref{dem-eq-os2}, we are led to \eqref{u_3-small0}.
\end{proof}
\begin{proof}[Proof of Proposition~\ref{teo-regularidad}]
Since $ u$ has finite energy, for every $\ve>0$, there exists $\rho >0$ such that for all $y\in \R^2$
\begin{equation}\label{unif}
E_{y,\rho}(u)\leq \ve.
\end{equation}
In fact, since $e(u)\in L^1(\R^2)$, by Lemma \ref{decomp-L1},
for every $\ve>0$ we can decompose
$e(u)=e_{1,\ve}(u)+e_{2,\ve}(u)$ such that
\begin{equation*}\norm{e_{1,\ve}(u)}_{L^1(\R^2)}\leq \ve/2\quad \textup{ and }\quad \norm{e_{2,\ve}(u)}_{L^\infty(\R^2)}\leq K_\ve,\end{equation*}
for some constant $K_\ve$. Then for any  $y\in \R^2$,
\begin{equation*}\norm{e_{2,\ve}(u)}_{L^1(B(y,\rho))}\leq K_\ve \pi \rho^2.\end{equation*}
Taking \begin{equation*}\rho=\left(\frac\ve{2K_{\ve}
\pi}\right)^{1/2},\end{equation*} we obtain \eqref{unif}.
Thus, invoking Corollary~\ref{cor-est}, with $\ve=\ve_0$ and $r=\min\{1,\rho\}$, we conclude
that
\begin{equation*}
\norm{\grad u}_{L^\infty(B(y,r/4))}\leq K(\ve_0)(1+c) E(u)^{1/4},\qquad \textup{ for all }y\in \R^2.
\end{equation*}
Therefore $u\in W^{1,\infty}(\R^2)$, with $\norm{\grad u}_{L^\infty(\R^2)}\leq K(\ve_0)(1+c)E(u)^{1/4}.$
Differentiating  \eqref{TW-LL}, we find
that $v =\ptl_j u$, $j=1,2$, satisfies
\begin{equation*}
  L_\lambda(v)\equiv-\Delta v-2(\grad u: \grad v )u-c (u\times \ptl_1 v)+\lambda v=\abs{\grad{u}}^2 v+2u_3v_3 u+u_3^2 v
-  v_3e_3 +
c (v\times \ptl_1 u)+
\lambda v, \end{equation*}
in $\R^2$. Since $\grad u\in L^\infty(\R^2)\cap L^2(\R^2)$, we deduce that the r.h.s. of the formula above belongs to $L^2(\R^2)$.
Therefore taking $\lambda>0$ large enough, we can invoke the elliptic regularity theory for linear systems and deduce that $v\in W^{2,2}(\R^2)$.
Then, by the Sobolev embedding theorem, $D^2u\in L^p(\R^2)$, for all $p\in [2,\infty)$
and a bootstrap argument allows us to conclude that $\grad u \in {W^{k,p}(\R^2)}$
for all $k\in \N$ and $p\in [2,\infty]$.

The estimates \eqref{u_3-small0} and \eqref{grad-E-small0} are given by Corollary~\ref{cor-est}.
\end{proof}

Proposition~\ref{teo-regularidad} shows that $u_3$ is uniformly continuous, so that
$u_3(x)\to0$, as $\abs{x}\to \infty$. In particular $u$ belongs to the space $\tilde \E(\R^2)$, where
\begin{equation*}
\tilde \E(\R^N)=\{v\in \E(\R^N) : \exists R\geq 0 \text{ s.t. } \norm{v_3}_{L^\infty(B(0,R))^c)}<1 \}.
\end{equation*}
In the case $N\geq 3$, we always suppose that $u \in \E(\R^N)\cap UC(\R^N)$ and then it is immediate that $u\in \tilde \E(\R^N)$.
Now we recall a well-known result (see e.g. \cite[Proposition 2.5]{maris-non}) that provides the existence of the lifting for any function in $\tilde \E(\R^N)$.
\begin{lemma}\label{lifting-tilde-E}
Let $N\geq 2$ and $v\in \tilde \E(\R^N)$. Then there exists $R\geq 0$ such that $v$ admits the lifting
\begin{equation}\label{eq-lif}
 \check v(x)=\varrho(x) e^{i \theta(x)}, \qquad \textup{ on } B(0,R)^c,\end{equation}
where $\vr=\sqrt{1-v_3^2}$ and $\theta$ is a real-valued function.
Moreover, $\varrho,\theta\in H_{\loc}^1(B(0,R)^c)$
and  $\grad \varrho,\grad\theta\in L^{2}(B(0,R)^c))$.
\end{lemma}
\begin{cor}\label{exis-lifting}
Let $c\geq 0$ and $u\in \E(\R^2)$ be a solution of \eqref{TW-LL}.
Then there is $R\geq 0$ such that the lifting $ \check u(x)=\varrho(x) e^{i \theta(x)}$
holds on $B(0,R)^c$ and satisfies $\grad \vr,\grad\theta\in W^{k,p}(B(0,R)^c)$
for any $k\geq 2$ and $p\in[2,\infty]$.
Moreover, there exists a constant $\ve(c)>0$, depending only on $c$, such that if $E(u)\leq \ve(c)$, then we can take $R=0$.
\end{cor}
\begin{proof}
By Proposition~\ref{teo-regularidad}, $u\in \tilde \E(\R^2)$ and then Lemma~\ref{lifting-tilde-E}
gives us the existence of the lifting, whose properties  follow from Proposition~\ref{teo-regularidad}
and the identity
\begin{equation}\label{iden-lif}
\abs{\grad \check u}^2=\vr^2 \abs{\grad \theta}^2+\abs{\grad \vr}^2, \qquad \textup{ on }B(0,R)^c,
\end{equation}
noticing that  $1-\norm{v_3}_{L^\infty(B(0,R)^c)}^2=\inf\{\vr(x)^2: x\in B(0,R)^c \} >0$.
The last assertion is an immediate consequence of \eqref{u_3-small0}.
\end{proof}

In the case $N\geq 3$, some regularity for the solutions of the equation \eqref{eq-g}
can be obtained considering that $u$ is a stationary solution in the sense introduced by
R.~Moser in \cite{moser}.
\begin{definition}
 Let $\Omega\subset \R^N$ be a smooth bounded domain and $g\in L^p(\Omega;\R^3)$.
A solution $u\in H^1(\Omega;\S^2)$ of \eqref{eq-g} is called {\em stationary} if
\begin{equation}\label{stationary}
\div(\abs{\grad u}^2 e_j-2 \grad u.\ptl_j u )=2 \ptl_j u\cdot g, \qquad \textup{ in } \Omega,
\end{equation}
for all $j\in\{1,\dots,N\}$ in the distributional sense.
\end{definition}

If we suppose that $u$ is a smooth solution of \eqref{eq-g}, then
\begin{equation*}
\div(\abs{\grad u}^2 e_j-2 \grad u.\ptl_j u )=-2\Delta u\cdot \ptl_j u=2 g\cdot  \ptl_j u,
\end{equation*}
so it is a stationary solution. However not every solution $u\in H^1(\Omega;\S^2)$  of
\eqref{eq-g} satisfies \eqref{stationary}. The advantage of stationary solutions is that they satisfy a monotonicity formula
that allows to generalize some standard results for harmonic maps. However, when
$g$ belongs only to $L^2(\Omega)$, the regularity estimates hold only for $N\leq 4$.

\begin{lema}[\cite{moser}]\label{lema-moser0}
Let $N\leq 4$ and $y\in \R^N$ . Assume that $u \in H^1(B(y,1);\S^2)\cap W^{1,4}(B(y,1))$
 is a stationary solution of \eqref{eq-g}, with $\Omega=B(y,1)$ and $g\in L^2(B(y,1))$.
Then there exist $K>0$ and $\ve_0>0$, depending only on $N,$ such that if
\begin{equation*}\norm{\grad u}_{L^2(B(y,1))}+\norm{g}_{L^2(B(y,1))}=\ve\leq \ve_0,\end{equation*}
we have
\begin{equation*}
\norm{\grad u}_{L^4(B(y,1/4))}\leq K\ve^\frac12.
\end{equation*}
\end{lema}

Applying this result to equation \eqref{TW-LL}, we are led to the following estimate.
\begin{lemma}\label{lema-moser}
Let $N\leq 4$. There exist $K>0$ and $\ve_0>0$, depending only on $N$, such that
for any solution $u\in \E(\R^N)\cap C^\infty(\R^N)$ of \eqref{TW-LL}, with $c\in[0,1]$, satisfying
$E(u)\leq \ve_0,$
we have
\begin{equation*}\norm{\grad u}_{L^4(B(x,1))}\leq K E(u)^{1/2}.\end{equation*}
\end{lemma}

Now we are in position to complete the regularity result in higher dimensions stated
in the introduction.

\begin{proof}[Proof of Proposition \ref{lema-reg-N}]
Recalling again a classical results
for elliptic systems with quadratic growth (see \cite{borchers,lady,jost}), $u \in UC(\R^N)$
yields that
$u\in C^\infty(\R^N)$. This is due to the fact that now we are assuming that $u$
is uniformly continuous and then we can choose $r>0$ small such that
the oscillation of $u$ on the ball $B(y,r)$ is small, uniformly in $y$.
Then we can make the quantity $A(u,r)$ defined in \eqref{cond-osc} as small as needed and
repeat the first part of the proof of Lemma~\ref{lemma-grad} to conclude that for all $y\in \R^N$
\begin{equation}\label{cota-D2u-2}
  \norm{D^2 u}_{L^2(B(y,r/2))}+\norm{\grad u}_{L^4(B(y,r/2))}^2\leq  K(N)r^{-1} \left(\norm{\grad u}_{L^2(B(y,r))}+\norm{u_3}_{L^2(B(y,r))}\right),
\end{equation}
for some constant $K(N)$ and $r>0$ small enough, independent of $y$.
At this stage we note that we cannot follow the rest of the argument of Lemma~\ref{lemma-grad},
since it relies on the two-dimensional Sobolev embeddings.
However, it is well-known  that using \eqref{cota-D2u-2}
it is possible to deduce that $\grad u \in L^p(\R^N)$, for all $p\geq 2$.
More precisely, as discussed before, there exists $r\in (0,1]$ such that
\begin{equation}\label{la-osc-r}
\osc_{B(y,2^{N}r)} u\leq \frac{1}{8(1+c)(2N-1)}, \qquad \textup{ for all }y\in \R^N.
\end{equation}
Then, by iterating Lemma~\ref{grad-lp}, we have
\begin{equation}\label{2N}
\int_{B(y,r)}\abs{\grad u}^{2N+2}\leq K(N)(1+c)^{2N} \frac{E(u)}{r^{2N}},\qquad \textup{ for all }y\in \R^N.
\end{equation}
By proceeding as in the proof of Lemma~\ref{lemma-grad}, we decompose $u_i$ as $u_i=\phi_i+\psi_i$,
where
\begin{equation*}
\begin{cases}
 -\Delta \phi_i=0, &\text{ in }  B(y,r),\\
\phi_i=u_i, &\text{ on } \ptl B(y,r),
\end{cases}
\end{equation*}
\begin{equation*}
\begin{cases}
 -\Delta \psi_i=\abs{\grad u}^2u_i+u_3^2 u_i-\delta_{i,3}u_3+cu\times \ptl_1 u, &\text{ in } B(y,r),\\
\psi_i=0, &\text{ on } \ptl B(y,r).
\end{cases}
\end{equation*}
In view of \eqref{2N},  elliptic regularity estimates imply that  $\psi_i\in W^{2,N+1}(B(y,r))$
and then by the Sobolev embedding theorem we can establish an upper bound for
$\norm{\grad \psi_i}_{L^\infty(B(y,r))}$ in terms of powers of $E(u)$.
Since $\phi_i$ is a harmonic function, we obtain a similar estimate for $\phi_i$
as in the proof of Lemma~\ref{grad-LL}. Then we conclude that $\grad u\in L^\infty(\R^N)$, so that,
 by interpolation, $\grad u\in L^p(\R^N)$, for all $p\in[2,\infty]$.
Proceeding as in the proof of Proposition~\ref{teo-regularidad}, we conclude that
$\grad u\in W^{k,p}(\R^N)$, for all $k\in \N$ and $p\in[2,\infty]$.

Now we turn to \eqref{est-N-3} and \eqref{est-N-3-grad}. Let us first take $N=3$ and $\ve_0$ given by Lemma~\ref{lema-moser},
such that $E(u)\leq \ve_0$. Then, by the Morrey inequality,
\begin{equation}\label{osc-N-3-3}
\osc_{B(y,1/2)} u\leq K  \norm{\grad u}_{L^4(B(y,1))}\leq K E(u)^{1/2},
\end{equation}
for all $y\in \R^3$ and for all $c\in [0,1]$. Taking $\ve_0$ smaller if necessary,  \eqref{la-osc-r} holds with $r=1/16$ and then so it does
 \eqref{2N} (with $r=1/16$). Hence the previous computations give a bound for $\grad u\in L^\infty(\R^3)$
depending only on $E(u)$, which yields \eqref{est-N-3-grad}.

In order to prove $\eqref{est-N-3}$, we estimate the minimum of $\abs{u_3}$ on ${B(y,1/2)}$ as in \eqref{dem-eq-os1},
and using \eqref{osc-N-3-3} we conclude that
 \begin{equation*}
 \max_{B(y,1/2)}{\abs{u_3}}\leq K E(u)^{1/2},
 \end{equation*}
which implies \eqref{est-N-3}.

It only remains to consider the case $N=4$.
Note that the r.h.s. of \eqref{la-osc-r} is less than or equal to $1/112$, for $c\in [0,1]$.
Let $r_*>0$ be the maximal radius given by the uniform continuity of $u$ for this value, i.e.
\begin{equation}
r_*=\sup\left\{\rho>0 : \forall x,z\in \R^4, \abs{x-z}\leq \rho\Rightarrow \abs{u(x)-u(z)}\leq 1/112\right\}.
\end{equation}
We claim that $r_*\geq 1/2$ for $\ve_0$ small. Arguing by contradiction, we suppose that $0<r_*<1/2$.
Since \eqref{app-osc} is satisfied for any $y\in\R^4$, with $r=r_*$ and $s=2$, Lemma~\ref{grad-lp} implies that
\begin{equation*}
\norm{\grad u}_{L^6(B(y,r_*/2))}^6\leq 8\left(1+\frac{16}{r_*^2}\right) \norm{\grad u}_{L^4(B(y,r_*))}^4.
\end{equation*}
Since $0<r_*<1/2$, the Morrey inequality implies that
\begin{equation}\label{osc-N-3}
  \osc_{B(y,r_*/4)} u\leq K_1 r_*^\frac13 \norm{\grad u}_{L^6(B(y,r_*/2))}\leq K_2 r_*^\frac13 \left(1+\frac{16}{r_*^2}\right)^\frac16
\norm{\grad u}_{L^4(B(y,r_*))}^\frac23\leq K_3 E(u)^\frac13,
\end{equation}
where we have used Lemma~\ref{lema-moser} for the last inequality and  $K_3>0$ is a universal constant.
Finally we notice that there exists a universal constant $\ell\in \N$ such that for any $x\in \R^4$, there is a collection of
points $y_1,y_2,\dots, y_\ell\in \R^4$ such that
\begin{equation*}\osc_{B(x,2r_*)}u\leq \sum_{k=1}^\ell \osc_{B(y_k,r_*/4)}u.\end{equation*}
Thus, using \eqref{osc-N-3}, $\osc_{B(x,2r_*)}u\leq \ell  K_3 E(u)^\frac13.$
Taking $ \ve_0\leq 1/(112\ell K_3)^3$, we get that $\osc_{B(x,2r_*)}\leq 1/112$, which contradicts the definition of $r_*$.
Therefore,
$\osc_{B(x,1/2)} u\leq 1/112$, for all $x\in \R^4.$
Moreover, the same argument shows that
\begin{equation*}\osc_{B(y,1/8)} u\leq KE(u)^{1/3},\qquad\textup{ for all }y\in \R^4,
 \end{equation*}
and then \eqref{est-N-3} and \eqref{est-N-3-grad} follow as before.
\end{proof}
\section{Properties related to the kernels and the convolution equations}\label{sec-conv}
Through this section, we  fix $u\in \tilde \E(\R^N)\cap UC(\R^N)$ a solution of \eqref{TW-LL} for a speed $c\in [0,1]$.
We also use the notation introduced in Subsection~\ref{subsec}.

We start recalling the following result for $L_c$.
\begin{lema}{\cite{delaire-cras,maris-chiron}}.
For any $c\in (0,1]$, we have
\begin{equation}\label{chiron}
\norm{L_c \wh f}_{L^{4/3}(\R^2)}\leq 11\norm{f}_{L^1(\R^2)},
\end{equation}
and
\begin{equation}\label{yo}
\norm{L_c \wh f}_{L^{2}(\R^N)}\leq K(N) \norm{f}_{L^{\frac{2(2N-1)}{2N+3}}(\R^N)},\qquad \text{if } N\geq 3.
\end{equation}
\end{lema}
\begin{proof}
By the Plancherel identity, the estimate \eqref{yo} is exactly \cite[Lemma 4.3]{delaire-cras}.
To prove \eqref{chiron}, we note that
 \begin{equation*}\norm{ L_c\wh f}_{L^{4/3}(\R^2)}\leq \norm{ L_c}_{L^{4/3}(\R^2)}\norm{\wh f}_{L^\infty(\R^2)}\leq
\norm{ L_c}_{L^{4/3}(\R^2)}\norm{f}_{L^1(\R^2)}.
 \end{equation*}
Then it only remains to compute $\norm{L_c}_{L^{4/3}(\R^2)}$. Using polar coordinates, we have
\begin{equation}\label{norma-calculo}
\begin{split}
   \norm{ L_c}_{L^{4/3}(\R^2)}^{4/3}&=4\int_0^\infty \int_0^{\pi/2}\frac{r\,d\theta\,dr}{(r^2+1-c^2\cos^2(\theta))^4/3}=
 6\int_0^{\pi/2}\frac{d\theta}{(1-c^2\cos^2(\theta))^{1/3}}\\
 &\leq  6\int_0^{\pi/2}\frac{d\theta}{(1-\cos^2(\theta))^{1/3}}=6\int_0^{\pi/2}\frac{d\theta}{\sin^{2/3}(\theta)}=3B\left(\frac16,\frac12\right),
  \end{split}
 \end{equation}
where $B$ denotes the Beta function. Using that  $B(x,y)=\Gamma(x)\Gamma(y)/\Gamma(x+y),$
we conclude that   \begin{equation}\label{norma-calculo2}
\norm{ L_c}_{L^{4/3}(\R^2)}\leq \left(3\frac{\Gamma(1/6)\Gamma(1/2)}{\Gamma(2/3)}\right)^{3/4}\leq 11.
\end{equation}
From \eqref{norma-calculo} and \eqref{norma-calculo2},  \eqref{chiron} follows.
\end{proof}

Now we are able to prove the exact form of estimate \eqref{E1} stated in the introduction and also
further integrability for $u_3$.

\begin{prop}\label{prop-u_3-N}
Let $N\geq 2$ and $c\in(0,1)$. Then $u_3\in L^p(\R^N)$, for all $p\in (1,2)$. Moreover, if $c\in (0,1]$ and
$\norm{u_3}_{L^\infty(\R^N)}\leq 1/2$, we have
\begin{equation}\label{estimacion-u_3-chiron}
\norm{u_3}_{L^4(\R^2)}\leq 54 \norm{u_3}_{L^\infty(\R^N)} E(u),
\end{equation}
and
\begin{equation}\label{estimacion-u_3}
  \norm{u_3}_{L^2(\R^N)}\leq K(N)\norm{u_3}_{L^\infty(\R^N)}\left(1+
\norm{\grad u}^{\frac{2N-5}{2(2N-1)}}_{L^{\infty}(\R^N)}\right)E(u)^\frac{2N+3}{2(2N-1)},\qquad \textup{if }N\geq 3.
\end{equation}
\end{prop}

\begin{proof}
Let us recall that by Propositions \ref{teo-regularidad} and \ref{lema-reg-N}, and noticing that
\begin{equation}\label{bon-G}
G=-u_3^2\grad \theta, \qquad \text{on }B(0,3R)^c,
\end{equation}
we infer that $F,G_1,G_2\in L^p(\R^N)$, for all $p \in [1,\infty]$.
On the other hand, from the Riesz-operator theory, the functions
$\xi\mapsto {\xi_i\xi_j}/{\abs{\xi}^2}$ are $L^q$-multipliers for any $q\in(1,\infty)$ and $1\leq i,j\leq N$.
 Since $L_c$ is also an $L^q$-multiplier for any $q\in(1,\infty)$ (see  \cite{gravejat-decay}), from \eqref{eq-fourier}
we conclude that $u_3\in L^q(\R^N)$, for all $q\in(1,\infty).$

We turn now to the proof \eqref{estimacion-u_3-chiron}. Using \eqref{eq-fourier} and the Hausdorff--Young inequality
  \begin{equation*}\norm{\eta}_{L^q(\R^N)}\leq p^{1/2p}q^{-1/2q} \norm{\wh \eta}_{L^p(\R^N)}, \quad p\in [1,2], \ q=p/(p-1),
     \end{equation*}
 with $p=4/3$, and \eqref{chiron} we obtain
\begin{equation}\label{dem-prop-u_3-1}
  \norm{ u_3}_{L^{4}(\R^2)}\leq \norm{ \wh u_3}_{L^{4/3}(\R^2)}\leq 11
\left( \norm{F}_{L^1(\R^2)}+\norm{G_1}_{L^1(\R^2)}+
\frac12 \norm{G_2}_{L^1(\R^2)}\right),
\end{equation}
where we have used that $\xi_1^2/\abs{\xi}^2\leq 1$ and $\xi_1\xi_2/\abs{\xi}^2\leq (\xi_1^2+\xi_2^2)/(2\abs{\xi}^2)\leq 1/2$ for the last inequality.

On the other hand, since $\norm{u_3}_{L^\infty(\R^N)}\leq 1/2$,
 the inequality \eqref{estim:polar} implies that
\begin{equation}\label{dem-prop-u_3-2}
\norm{G_j}_{L^1(\R^N)}\leq \frac2{\sqrt3} \norm{u_3}_{L^\infty(\R^N)}E(u).
\end{equation}

From \eqref{dem-prop-u_3-1} and \eqref{dem-prop-u_3-2}, since $F=2e(u)u_3+cG_1$
and $11(2+5/\sqrt{3})<54$, \eqref{estimacion-u_3-chiron} follows.

Let us prove now  \eqref{estimacion-u_3}. By applying the Plancherel identity to \eqref{eq-fourier} and using \eqref{yo} and
\eqref{dem-prop-u_3-2}, we are led to
\begin{equation*}
\begin{split}
\norm{u_3}_{L^2(\R^N)}&\leq K(N)
\Bigg(\norm{F}_{L^{\frac{2(2N-1)}{2N+3}}(\R^N)}+\sum_{j=1}^N \norm{G_j}_{L^{\frac{2(2N-1)}{2N+3}}(\R^N)}\Bigg)\\
&\leq K(N) \norm{u_3 e(u)}_{L^{\frac{2(2N-1)}{2N+3}}(\R^N)}\\
&\leq K(N) \norm{u_3}_{L^{\infty}(\R^N)}\norm{e(u)}_{L^\infty(\R^N)}^\frac{2N-5}{2(2N-1)}\norm{e(u)}_{L^1(\R^N)}^\frac{2N+3}{2(2N-1)},\\
\end{split}
\end{equation*}
which gives \eqref{estimacion-u_3}.
\end{proof}
\begin{lema}\label{integ-Lp}
For all $k\in \N$ and $p\in(1,\infty]$, we have  $u_3,\grad(\chi \theta)\in W^{k,p}(\R^N)$.
\end{lema}
\begin{proof}
 By Propositions~\ref{teo-regularidad} and \ref{lema-reg-N}, it remains only to treat the case $p\in(1,2)$.
Differentiating \eqref{conv-u3} and \eqref{conv-theta}, we have
\begin{align*}
 \ptl^\alpha u_3&=\mathcal L_c* \ptl^\alpha F-c \sum_{j=1}^N\mathcal L_{c,j}*  \ptl^\alpha G_j,\\
 \ptl^\alpha \ptl_j(\chi \theta)&=c\, \boL_{c,j}*\ptl^\alpha F-c^2\sum_{k=1}^N \boT_{c,j,k}*\ptl^\alpha G_k-\sum_{k=1}^N \boR_{j,k}* \ptl^\alpha G_k,
\end{align*}
for all $\alpha \in \N^N$.
The conclusion follows by observing that $\boL_{c,j},\boT_{c,j,k}$ and $\boR_{j,k}$ are $L^p$-multipliers
for all $p\in(1,\infty)$, that $u_3,\grad(\chi \theta),\grad u \in W^{k,p}(\R^N)$ for all $k\in \N$ and $p\in[2,\infty)$ and using
the Leibniz rule.
\end{proof}

\begin{cor}\label{lema-theta-infty}
Let $N\geq 2$ and $c\in [0,1)$. Then the function $\theta$ is bounded on $B(0,R)^c$ and there exists $\bar \theta\in \R$
such that
\begin{equation}\label{limite-theta}
\theta(x)\to\bar \theta,\qquad \textup{as }\abs{x}\to\infty.
\end{equation}
\end{cor}
\begin{proof}
By Lemma~\ref{integ-Lp}, $\grad \theta\in L^p(\R^N)$, for all $1<p\leq\infty$.
Then there exists $\bar \theta\in \R$
such that $\theta-\bar \theta \in L^{\frac{Np}{N-p}}(\R^N)$ (see e.g. \cite[Theorem 4.5.9]{hormander}).
Since   $\grad \theta\in L^\infty(\R^N)$, we have $\theta\in UC(\R^N)$ and therefore \eqref{limite-theta} follows.
\end{proof}

\begin{proof}[Proof of Proposition~\ref{prop-static-L}]
For $c=0$, we deduce from \eqref{poh2-LL} and  \eqref{poh1-LL} that $\norm{u_3}_{L^2(\R^N)}=0$, so that $u_3\equiv 0$.
 Thus $\check u=e^{i\theta}$ on $\R^N$ and using \eqref{TW-LL} (see \eqref{polar-1}) we deduce that $\Delta \theta=0$ on $\R^N$. Therefore, by
Corollary~\ref{lema-theta-infty},  we obtain that  $\theta$ is a bounded harmonic function,
which implies that it is constant and so that $\check u$ is constant.
\end{proof}
\section{Pohozaev identities}\label{sec-poh}
We start establishing the following Pohozaev identities for \eqref{TW-LL}.
For this purpose, we introduce the notation
\begin{equation*}w_k(v)\equiv v\cdot (\ptl_1 v\times \ptl_k v),\qquad k\in \{2,\dots,N\}.
 \end{equation*}

\begin{prop}\label{pohozaev-LL}
 Let $u\in  \E(\R^N)\cap C^2(\R^N)$ be a solution of \eqref{TW-LL}. Then
there exists a sequence $r_n\to \infty$ such that
 \begin{align}
E(u)&=\intRR{\abs{\ptl_1 u}^2}\,dx\label{poh2-LL},\\
E(u)&=\intRR{\abs{\ptl_k u}^2}\,dx-c\lim_{r_n\to\infty}\int_{B(0,r_n)} x_k w_k(u)\label{poh1-LL}\,dx, \qquad  \textup{ for all }k\in\{2,\dots,N\}.
\end{align}
\end{prop}
\begin{proof}
Taking  inner product between \eqref{TW-LL} and $x_k\ptl_ku$, $1\le k\le N$, integrating by parts
in the ball $B(0,R)$ and using that $u\cdot \ptl_k u =0$, we obtain
\begin{equation*}
\begin{split}
  \int_{B(0,R)}\abs{\ptl_k u}^2
  -\frac{1}2\int_{B(0,R)}\abs{\grad u}^2
  -\int_{\ptl B(0,R)}\frac{\ptl u}{\ptl \nu}\cdot \ptl_k u x_k
+\int_{\ptl B(0,R)}\abs{\grad  u}^2 x_k \nu_k
=\\ \frac12 \int_{B(0,R)}u^2_3-\frac12\int_{\ptl B(0,R)}u^2_3 x_k \nu_k
+c\int_{B(0,R)}x_k w_k(u),
\end{split}
\end{equation*}
where $\nu$ denotes the exterior normal of the ball $B(0,R)$ and $\frac{\ptl u}{\ptl \nu}=(\grad u_1\cdot \nu,\grad u_2\cdot \nu,\grad u_3\cdot \nu)$.
By Lemma~\ref{est-ipp}, there is a sequence $r_n\to \infty$ such that
\begin{equation*}
-\int_{\ptl B(0,r_n)}\frac{\ptl u}{\ptl \nu}\cdot \ptl_k u x_k
+\int_{\ptl B(0,r_n)}\abs{\grad  u}^2 x_k \nu_k+\frac12\int_{\ptl B(0,r_n)}u^2_3 x_k \nu_k\to 0, \quad \text{ as }n\to\infty.
\end{equation*}
Therefore
\begin{equation*}
  E(u)=\intRR{\abs{\ptl_k u}^2}-c\lim_{r_n\to\infty}\int_{B(0,r_n)} x_k w_k(u),
\end{equation*}
which completes the proof.
\end{proof}

Let us now discuss the definition of momentum in the two dimensional case. Formally, the first component of the vectorial momentum is given by (see \cite{papanicolaou})
\begin{equation*}
p(v)=-\int_{\R^2} x_2  w_2(u)\,dx,
\end{equation*}
but it is not clear that this quantity is well-defined in $\E(\R^2)$.
In general, it is a delicate task to define the momentum as a functional is the energy space.
This difficulty also appears in the context of the Gross--Pitaevskii equation (see e.g. \cite{delaire}). For the purpose of this paper,
we will only define  $p$  for smooth solutions of \eqref{TW-LL}.
In fact, from Proposition~\ref{pohozaev-LL},  there exists a sequence $r_n\to \infty$ such that
the limit
\begin{equation*}
\lim_{r_n\to \infty} \int_{B(0,r_n)} x_2  x_2  w_2(u)\,dx,
\end{equation*}
exists. Moreover, \eqref{poh1-LL} shows that this limit does not depend on the sequence $r_n$ and therefore we will define this quantity as the momentum
\begin{equation}\label{la-aprox}
p(u)=-\lim_{r_n\to \infty} \int_{B(0,r_n)} x_2  w_2(u)\,dx.
\end{equation}
With this notation we have the following consequence of Proposition~\ref{pohozaev-LL}.

\begin{cor}\label{cor-poh3} Let $u\in \E(\R^2)$ be a solution of \eqref{TW-LL}. Then
\begin{equation}\label{poh3}
 \intRx{u_3^2}=cp(u).
\end{equation}
\end{cor}
\begin{proof}
Writing
 \begin{equation*}\intRx{u_3^2}=2E(u)-\intRx{\abs{\ptl_1 u}^2}-\intRx{\abs{\ptl_2 u}^2},
   \end{equation*}
since $u\in C^2(\R^2)$ by Proposition~\ref{teo-regularidad},
the result is a direct consequence of  Proposition~\ref{pohozaev-LL}.
\end{proof}

In the case that $u$ admits a global lifting, we obtain
\begin{lemma}
Let $u\in  \E(\R^2) \cap  C^2(\R^2)$ such that $\norm{u_3}_{L^\infty(\R^2)}<1$.
Then
\begin{equation}\label{mom-lif}
p(u)=\int_{\R^2}{u_3\ptl_1\theta},
\end{equation}
where $u_1+iu_2=\sqrt{1-u_3^2}e^{i\theta}$.
\end{lemma}
\begin{proof}
First we notice that
 \begin{equation}
 \label{estim:polar}
|u_3 \partial_1 \theta| \leq \frac{|u_3| |1 - u_3^2|^\frac{1}{2} |\partial_1 \theta|}{({1 - \norm{u_3}_{L^\infty(\R^2)}^2})^{1/2}}  \leq
\frac{e(u)}{({1 - \norm{u_3}_{L^\infty(\R^2)}^2})^{1/2}},
\end{equation}
so that the integral in \eqref{mom-lif} is well-defined in $\E(\R^2)$.
We notice that
\begin{equation*}
  u\cdot (\ptl_1u\times \ptl_2 u)=u_3\vr (\ptl_1 \vr\ptl_2 \theta-\ptl_2\vr \ptl_1\theta )+
\vr^2(\ptl_1\theta \ptl_2 u_3-\ptl_2 \theta \ptl_1 u_3)=\ptl_2 (u_3\ptl_1\theta)-\ptl_1( u_3\ptl_2\theta),
\end{equation*}
where we have used that $u_3^2=1-\vr^2$ for the last equality.
Then, multiplying by $x_2$, integrating by parts and using the definition of $p(u)$,  \eqref{mom-lif} follows.
\end{proof}

From \eqref{estim:polar}, we see that integral in \eqref{mom-lif} is well-defined in $\E(\R^2)$. Actually, integrating on $\R^N$ instead of $\R^2$, this expression provides
a general definition of momentum, for functions that admit a global lifting, in any dimension. We will see this for $N=1$ in Section~\ref{sec-N-1}.


\section{Properties of solutions satisfying $\norm{u_3}_{L^\infty(\R^N)}\leq 1/2$}\label{sec-vortexless}
In this section we assume that $u\in \tilde\E(\R^N)\cap UC(\R^N)$ is a nontrivial solution of \eqref{TW-LL}
with $c\in (0,1]$ and
 \begin{equation}\label{vortexless}
\norm{u_3}_{L^\infty(\R^N)}\leq \frac12.
 \end{equation}
We have chosen $1/2$ to simplify the estimates. The main assumption here is that
$\norm{u_3}_{L^\infty(\R^N)}<1$, which implies that $\check u=\varrho e^{i\theta}$ on  $\R^N$. Hence we can recast \eqref{TW-LL}  as
 \begin{align}
 &\div(\varrho^2 \grad \theta)=c\ptl_1 u_3,\label{polar-1}\\
 &-\Delta\varrho +\varrho\abs{\grad \theta }^2=2e(u)\varrho -cu_3\varrho \ptl_1\theta\label{polar-2},\\
  &-\Delta u_3= (2e(u)-1) u_3+c\varrho^2\ptl_1\theta.\label{polar-3}
\end{align}
From these equations we obtain the following useful integral relations.
\begin{lema} We have the following identities
 \begin{align}
\intRR \varrho^2 \abs{\grad \theta }^2&=c\intRR u_3 \ptl_1\theta,\label{dem-int-1}\\
\intRR \abs{\grad \varrho}^2+\intRR \varrho^2 \abs{\grad \theta}^2&=2\intRR e(u) \varrho^2-c\intRR u_3\varrho^2 \ptl_1 \theta,\label{dem-int-2}\\
 2\intRR \vr \abs{\grad \varrho}^2+2\intRR e(u)u_3^2\varrho &=\intRR \vr u_3^2\abs{\grad \theta}^2+c\intRR \vr u_3^3\ptl_1 \theta,
 \label{dem-int-2bis}\\
\intRR \abs{\grad u_3}^2+\intRR u_3^2&=2\intRR e(u)u_3^2+c\intRR \varrho^2 u_3 \ptl_1\theta. \label{dem-int-3}
\end{align}
\end{lema}
\begin{proof}
First we recall that by Lemma~\ref{est-ipp}, for any $f\in L^2(\R^2),$
there exists a sequence $R_n\to \infty$ such that
\begin{equation}\label{int-borde}
 \int_{\ptl B(0,R_n)}\abs{f}\leq \frac{(2\pi)^{1/2}}{(\ln(R_n))^{1/2}}.
\end{equation}
Now we multiply \eqref{polar-1} by $\theta$ and integrate by parts on the ball $B(0,R_n)$. Using the fact that $u_3,\grad \theta\in L^2(\R^N)$
and $u_3,\vr,\theta \in L^\infty(\R^N)$, we can choose $R_n$ as in \eqref{int-borde} such that the integrals on $\ptl B(0,R_n)$
go to zero and \eqref{dem-int-1} follows.

To obtain \eqref{dem-int-2}, \eqref{dem-int-2bis} and \eqref{dem-int-3}, we
 multiply  \eqref{polar-2} by $\varrho$, \eqref{polar-2} by $u_3^2$ and  \eqref{polar-3} by $u_3$,
and proceed in a similar way.
\end{proof}

The following result corresponds to the estimate \eqref{E2} in the case $N\geq 3$.
\begin{prop}\label{prop-51}
\begin{equation}\label{la-otra-cota-N}
E(u)\leq  3\norm{u_3}^2_{L^2(\R^N)}.
\end{equation}
\end{prop}

\begin{proof}
Let $\delta=\norm{u_3}_{L^\infty(\R^N)}\in [0,1/2]$.
By the Cauchy--Schwarz inequality we have
\begin{equation*}
  \intRR u_3 \ptl_1\theta \leq \left(\intRR u_3^2\right)^\frac12 \left(\intRR (\ptl_1\theta) ^2\right)^\frac1{2}
 \leq \frac{1}{\sqrt{1-\delta^2}} \left(\intRR u_3^2\right)^\frac12 \left(\intRR \varrho^2 \abs{\grad \theta}^2\right)^\frac12.
\end{equation*}
Thus from \eqref{dem-int-1},
\begin{equation}\label{dem-cota-1-N}
\intRR \rho^2\abs{\grad \theta}^2\leq \frac{c^2}{1-\delta^2}\intRR u_3^2\leq \frac{4}{3} \intRR u_3^2.
\end{equation}
On the other hand, from \eqref{dem-int-2bis} and the Cauchy--Schwarz inequality, we obtain
\begin{equation}\label{tb-en-N}
\begin{split}
  2\sqrt{1-\delta^2}\intRR\abs{\grad \vr}^2 &\leq \frac{\delta^2}{\sqrt{1-\delta^2}}\intRR \vr^2\abs{\grad \theta}^2+c\delta^2
 \left(\intRR u_3^2\right)^\frac12 \left(\intRR \varrho^2 \abs{\grad \theta}^2\right)^\frac12.\\
\end{split}
\end{equation}
By combining \eqref{dem-cota-1-N} and \eqref{tb-en-N}, with $\delta\leq 1/2$, we are led to
\begin{equation}\label{dem-cota-2-N}
\intRR\abs{\grad \rho}^2\leq \frac{7}{18}\intRR u_3^2.
\end{equation}
From  \eqref{dem-int-3}, using the Cauchy--Schwarz inequality, we have
\begin{equation*}
  (1-\delta^2)\intRR(\abs{\grad u_3}^2+u_3^2)
\leq \delta^2\left( \intRR\abs{\grad \rho}^2+\intRR \rho^2\abs{\grad \theta}^2\right)+\left(\intRR u_3^2\right)^\frac12 \left(\intRR \varrho^2 \abs{\grad \theta}^2\right)^\frac12,
\end{equation*}
so that, by \eqref{dem-cota-1-N} and \eqref{dem-cota-2-N},
\begin{equation}\label{dem-cota-3-N}
\intRR\abs{\grad u_3}^2\leq\left( \frac43 \left(\frac14\left(\frac7{18}+\frac43\right)+\left(\frac43\right)^{1/2}\right)-1\right)\intRR u_3^2\leq 2\intRR u_3^2.
\end{equation}
Finally, by putting together \eqref{dem-cota-1-N}, \eqref{dem-cota-2-N} and \eqref{dem-cota-3-N},
\begin{equation*}
E(u)\leq \frac12\left( \frac43+\frac7{18}+3\right) \intRR u_3^2\leq 3\intRR u_3^2.
\end{equation*}
\end{proof}
In the two-dimensional case, Corollary~\ref{cor-poh3} allows us also to estimate the energy in terms of
 $\norm{u_3}_{L^4(\R^2)}$.
To this purpose, as remarked in \cite{maris-chiron}, it is useful to study the norm of $\ptl_1\theta-c u_3$.

\begin{lemma}\label{cota-phi}
\begin{equation*} \norm{\ptl_1\theta-c u_3}_{L^2(\R^2)}^2+\norm{\ptl_2\theta}_{L^2(\R^2)}^2\leq  \frac94 \norm{u_3}^4_{L^4(\R^2)}.\end{equation*}
 \end{lemma}
\begin{proof} By adding  \eqref{poh3} and \eqref{dem-int-1}, we obtain
\begin{equation}\label{la-suma}
\intR{u_3^2}+\intR \vr^2\abs{\grad \theta}^2=2c\intR{u_3\ptl_1\theta}.
\end{equation}
Since $\abs{\grad \theta}^2=(\ptl_1\theta)^2+(\ptl_2\theta)^2$,
by defining the function $\vp=\ptl_1\theta-c u_3$, we have $\ptl_1\theta=\vp+cu_3$
and then we recast \eqref{la-suma} as
\begin{equation*}
\intR (1-u_3^2)(\vp^2+(\ptl_2\theta)^2)+(1-c^2)\intR u_3^2=2c\intR u_3^3\vp+c^2\intR u_3^4.
\end{equation*}
Letting $\delta=\norm{u_3}_{L^\infty(\R^2)}$ and using that $c\in[0,1]$, we conclude that
\begin{equation}\label{lema-phi}
(1-\delta^2)(\norm{\vp}_{L^2(\R^2)}^2+\norm{\ptl_2\theta}_{L^2(\R^2)}^2)\leq 2\norm{u_3^3\vp}_{L^1(\R^2)}+\norm{u_3}^4_{L^4(\R^2)}.
\end{equation}
For the first term in the r.h.s., we use the H\"older inequality
\begin{equation}\label{para-cota}
\norm{u_3^3\vp}_{L^1(\R^2)}\leq \norm{u_3}_{L^\infty(\R^2)}\norm{u_3^2\vp}_{L^1(\R^2)}\leq
\delta\norm{u_3^2}_{L^2(\R^2)}\norm{\vp}_{L^2(\R^2)}.
\end{equation}
Then \eqref{lema-phi},   \eqref{para-cota} and the inequality $ab\leq a^2/4+b^2$ imply that
\begin{equation*}
\left(1-\delta^2-\frac{\delta}{2}\right)\norm{\vp}_{L^2(\R^2)}^2 +(1-\delta^2)\norm{\ptl_2\theta}_{L^2(\R^2)}^2\leq
\left(1+\frac{\delta}{4}\right)\norm{u_3}^4_{L^4(\R^2)}.
\end{equation*}
Since $\delta\leq1/2$, the conclusion follows.
\end{proof}

\begin{lemma}\label{cota-ro}
\begin{equation}\label{dem-cota-ro1} \norm{\grad \vr}_{L^2(\R^2)}^2\leq 6 \norm{u_3}^4_{L^4(\R^2)}.\end{equation}
 \end{lemma}
\begin{proof}
Since $\vr=\sqrt{1-u_3^2}\in [\sqrt{3}/2,1]$, from \eqref{dem-int-2bis} we have
\begin{equation}\label{dem-cota-ro7}
 \sqrt{3}\intR  \abs{\grad \varrho}^2+\sqrt{3}\intR e(u){u_3^2} \leq \intR \vr u_3^2( \abs{\grad \theta}^2+c  u_3\ptl_1 \theta).
\end{equation}
As in the proof of Lemma~\ref{cota-phi}, we define $\vp=\ptl_1\theta-c u_3$ so that
\begin{equation}\label{dem-cota-ro2}
\abs{\grad \theta}^2+c  u_3\ptl_1 \theta=\vp^2+3cu_3\vp+2c^2u_3^2+(\ptl_2\theta)^2.
\end{equation}
Since $\vr\leq 1$, $\abs{u_3}\leq 1$ and $c\leq 1$, using  \eqref{dem-cota-ro2} and the Cauchy--Schwarz inequality,
\begin{align}
  \intR \abs{\vr u_3^2( \abs{\grad \theta}^2+c  u_3\ptl_1 \theta)}
&\leq \norm{\vp}^2_{L^2(\R^2)}+3\norm{u_3}^2_{L^4(\R^2)}\norm{\vp}_{L^2(\R^2)}
+2\norm{u_3}^4_{L^4(\R^2)}+\norm{\ptl_2\theta}^2_{L^2(\R^2)}\nonumber\\
&\leq \frac{35}{4}\norm{u_3}^4_{L^4(\R^2)}, \label{dem-cota-ro3}
\end{align}
where we have used Lemma~\ref{cota-phi} for the last inequality.
By combining \eqref{dem-cota-ro7},  \eqref{dem-cota-ro3} and the fact that $35/(4\sqrt{3})\leq 6$, we obtain \eqref{dem-cota-ro1}.
\end{proof}

Finally, we get estimate \eqref{E2} for $N=2$.
\begin{prop}\label{E-L-4}
 \begin{equation*}E(u)\leq 10 \norm{u_3}^4_{L^4(\R^2)}.
   \end{equation*}
\end{prop}
\begin{proof}
From \eqref{dem-int-2} we obtain
\begin{equation*}
2\intR e(u) \varrho^2=
\intR  \abs{\grad \varrho}^2+\intR \vr^2 ( \abs{\grad \theta}^2+c  u_3\ptl_1 \theta).
\end{equation*}
By using Lemma~\ref{cota-ro}, \eqref{dem-cota-ro3} and the fact that $\vr^2\in[3/4,1]$,
we conclude that
 \begin{equation*}E(u)\leq \frac{59}{6}\norm{u_3}^4_{L^4(\R^2)}\leq 10\norm{u_3}^4_{L^4(\R^2)}.
   \end{equation*}
\end{proof}
At this point we dispose of all the elements to prove our result, as was sketched in the introduction.
\begin{proof}[Proof of Theorem~\ref{teo-non-N}]
By virtue of Propositions~\ref{teo-regularidad} and  \ref{lema-reg-N},
 we can fix $\ve_0>0$
such that if \mbox{$E(u)\leq \ve_0$}, then $\norm{u_3}_{L^\infty(\R^N)}\leq 1/2$
and $\norm{\grad u}_{L^\infty(\R^N)}$ is uniformly bounded. Then Propositions~\ref{prop-u_3-N}, \ref{prop-51} and \ref{E-L-4} imply that
\begin{equation}\label{fin-dem1}
E(u)\leq 10\norm{u_3}_{L^4(\R^N)}^4 \leq 10(54)^4 E(u)^4,  \quad \text{ if }N=2,
\end{equation}
and
\begin{equation}\label{fin-dem2}
E(u)\leq 3\norm{u_3}^2_{L^2(\R^N)}\leq K E(u)^{\frac{2N+3}{2N-1}}, \qquad \text{ if } N\in\{3,4\}.
\end{equation}
Thus, since $u$ is nonconstant, $E(u)>0$ and we can divide by $E(u)$. Therefore, from \eqref{fin-dem1} and \eqref{fin-dem2}
we conclude that $\tilde K\leq E(u)$, for some constant $\tilde K>0$. Taking $\mu=\min\{\ve_0,\tilde K\}$, the proof is complete.
\end{proof}
\begin{section}{The one-dimensional case}\label{sec-N-1}
In this section we consider the case $N=1$. Then equation \eqref{TW-LL}  is integrable  and the solutions can be
computed explicitly as was noticed in \cite{mikeska,sasada,sasada2}. More precisely, we have

\begin{prop}\label{sol-N-1}
Let $N=1$, $c\geq 0$ and $u\in \E(\R)$ be solution of \eqref{TW-LL}.
\begin{enumerate}
\item[(i)] If $c\geq 1$, then $u$ is a trivial solution.
\item[(ii)] If $0\leq c <1$ and $u$ is nontrivial, then, up to invariances, $u$ is given by
\begin{align}
u_1&=c \sech(\sqrt{1-c^2}\,x),\label{N1-u1}\\
u_2&=\tanh(\sqrt{1-c^2}\,x),\label{N1-u2}\\
 u_3&=\sqrt{1-c^2}\sech(\sqrt{1-c^2}\,x).\label{N1-u3}
\end{align}
\item[(iii)] If $0<c<1$, we can write
\begin{equation}\label{N-1-u3}
\check u= \sqrt{1-u_3^2}\exp(i \theta),
\end{equation}
where
\begin{equation}\label{theta}
\theta=\arctan\left(\frac{\sinh(\sqrt{1-c^2}\,x)}{c}\right).
\end{equation}
\end{enumerate}
\end{prop}

\begin{proof}
We first remark that since $N=1$, it is simply to verify that $u$ is smooth
and then the  condition $u\in\E(\R)$ implies that $u'$ and $u_3$ vanish at infinity.
Let us write \eqref{TW-LL} in coordinates
\begin{align}
 - u''_1&=2e(u) u_1+c(u_2 u'_3-u_3u'_2)\label{v1},\\
 -u''_2&=2e(u) u_2+c(u_3 u'_1-u_1u'_3)\label{v2},\\
-u''_3&=2e(u) u_3-u_3+c(u_1 u_2'-u_2 u'_1). \label{v3}
\end{align}
Also, as in \eqref{eq-div-G}, we have
\begin{equation}\label{dem-1-1}
(u_1u_2'-u_1'u_2)'=cu'_3.
\end{equation}
Integrating \eqref{dem-1-1}, we obtain
\begin{equation}\label{dem-1-2}
u_1u_2'-u_1'u_2=cu_3.
\end{equation}
Then, replacing \eqref{dem-1-1} in \eqref{v3}, we get
\begin{equation}\label{dem-1-3}
u_3''+2e(u)u_3-(1-c^2)u_3=0.
\end{equation}
Now, multiplying \eqref{v1}, \eqref{v2}, \eqref{dem-1-3} by $u_1',u_2',u_3'$, respectively, adding these relations
and using again \eqref{dem-1-2},
\begin{equation}\label{dem-1-4}
-(\abs{u'}^2)'=2e(u)(u_1^2+u_2^2+u_3^2)'-(u_3^2)'.
\end{equation}
Since  $(u_1^2+u_2^2+u_3^2)'=(\abs{u}^2)'=0$, integrating \eqref{dem-1-4}
we conclude that
\begin{equation}\label{ci}
\abs{u'}^2=u_3^2,
\end{equation}
so that  $e(u)=u_3^2$ and equation \eqref{dem-1-3} reduces to
\begin{equation}\label{dem-1-5}
u_3''-2u_3^3-(1-c^2)u_3=0.
\end{equation}
As before, multiplying \eqref{dem-1-5} by $u_3'$ and integrating,  we conclude that
\begin{equation}\label{dem-1-6}
(u_3')^2=u_3^2((1-c^2)-u_3^2).
\end{equation}
If  $u_3$ is identically zero, \eqref{ci} implies that $u$ is a trivial solution. Therefore, we suppose from now on
that $u_3$ not identically zero. Since equation \eqref{dem-1-5} is invariant under translation, we can assume that
\begin{equation*}
\abs{u_3(0)}=\max\{\abs{u_3(x)} : x\in \R\}>0.
\end{equation*}
Therefore
\begin{equation}\label{der-u-3}
u_3'(0)=0,
\end{equation}
and from \eqref{dem-1-6} and \eqref{der-u-3}, $u_3^2(0)={1-c^2}.$
In particular we deduce that if $c\geq 1$, $u_3\equiv 0$, which implies that $u_1$ and $u_2$ are constant,
which completes the proof of (i). If $0\leq c<1$, by the Cauchy--Lipschitz theorem, equation \eqref{dem-1-5} with initial conditions  \eqref{der-u-3}
and $u_3(0)=\sqrt{1-c^2}$ or $u_3(0)=-\sqrt{1-c^2}$ has a unique maximal solution. It is straightforward to check that
\begin{equation}\label{formula-u3}
u_3(x)=\pm \sqrt{1-c^2}\sech(\sqrt{1-c^2}\,x)
\end{equation}
is the desired solution. Moreover, \eqref{formula-u3} shows that $\norm{u_3}_{L^\infty(\R)}<1$ if $c\in (0,1)$.
Hence, for $c\in (0,1)$,
we can write
 $\check u=(1-u_3^2)^{1/2}e^{i\theta},$
and then  \eqref{dem-1-1} yields
\begin{equation}\label{dem-u-6}
\theta'=\frac{c u_3}{1-u_3^2}.
\end{equation}
From \eqref{formula-u3} and \eqref{dem-u-6}, we are led to
\begin{equation*}
\theta=\theta_0+\arctan\left(\frac{\sinh(\sqrt{1-c^2}\,x)}{c}\right),
\end{equation*}
for some constant $\theta_0\in \R$, which proves \eqref{N-1-u3}--\eqref{theta}. Using some standard identities
for trigonometric and hyperbolic functions, we also obtain \eqref{N1-u1}--\eqref{N1-u3}, for $c\in (0,1)$.
It only remains to show that for $c=0$, \eqref{N1-u1} and \eqref{N1-u2} are the unique solutions of \eqref{v1}--\eqref{v3}.
Indeed, since $e(u)(x)=u_3^2(x)=\sech^2(x)$,  we recast \eqref{N1-u1} and \eqref{N1-u2} as
\begin{equation}\label{dem-1-7}
-\check u''=2\sech^2(x)\check u,
\end{equation}
and from \eqref{ci} we can assume that, up to a multiplication by a complex number of modulus one,
\begin{equation}\label{dem-1-8}
\check u'(0)=1.
\end{equation}
Then the Cauchy--Lipschitz theorem provides the existence of a unique solution of
\eqref{dem-1-7}--\eqref{dem-1-8} in a neighborhood of $x=0$, and it is immediate to check
that $\check u(x)=\tanh(x)$ is the solution, which concludes the proof.
 \end{proof}

In the one-dimensional case, the momentum is formally given by 
 \begin{equation*}p(u)=\int_{\R}\frac{u_3(u_1u_2'-u_2u_1')}{1-u_3^2}.
   \end{equation*}
If $\norm{u_3}_{L^\infty(\R)}<1$, we see that
\begin{equation*}
p(u)=\int_{\R} u_3 \theta',
\end{equation*}
and therefore it agrees with the corresponding expression in dimension two.

\begin{cor}\label{cor-N-1} Assume that $c\in [0,1)$ and let $u\in \E(\R)$ be a nontrivial solution of \eqref{TW-LL}. Then
\begin{equation}\label{N1-E}
E(u) =2\sqrt{1-c^2}.
\end{equation}
Moreover,
\begin{equation}\label{N1-mom}
p(u)= \int_{\R}u_3\theta'=2\arctan\left(\frac{\sqrt{1-c^2}}{c}\right),\qquad  \textup{for }c\in(0,1).
\end{equation}
In particular, we can write explicitly $E$ as a function of $p$ as
\begin{equation}\label{N1-E-p}
E(p)=2\sin\left(\frac{p}2\right)
\end{equation}
and
\begin{equation}\label{la-der}
\frac{dE}{dp}=\cos\left(\frac{p}2\right)=c,
\end{equation}
for  $c\in(0,1)$.
\end{cor}
\begin{proof}
Using \eqref{N1-u3} and  \eqref{ci}, we have
\begin{equation*}
E(u)=\int_{\R}u_3^2=\sqrt{1-c^2}\int_{\R}\sech^2(x)\,dx=2\sqrt{1-c^2}.
\end{equation*}
For the momentum, \eqref{dem-u-6} yields
\begin{equation*}
  p(u)= \int_{\R}u_3\theta'=c\int_{\R}\frac{u_3^2}{1-u_3^2}=c(1-c^2)\int_{\R}\frac{\sech^2(\sqrt{1-c^2}x)}{1-(1-c^2)\sech^2(\sqrt{1-c^2}x)}dx.
\end{equation*}
Then, using the change of variables $y=\frac{\sqrt{1-c^2}}{c}\tanh(\sqrt{1-c^2}\,x)$, we obtain \eqref{N1-mom}, from where
we deduce that
\begin{equation}\label{c-p}
c^2=\frac{1}{\tan^2(p/2)+1}=\cos^2(p/2).
\end{equation}
Finally, from \eqref{N1-E} and \eqref{c-p}, we establish \eqref{N1-E-p}, from where \eqref{la-der} is an immediate consequence.
\end{proof}
Proposition~\ref{sol-N-1-intro} follows from Proposition \ref{sol-N-1} and Corollary \ref{cor-N-1}.

\end{section}

\begin{section}{Decay at infinity}\label{sec-decay}
In this section we provide a sketch the proof of Theorem~\ref{limite-infty}.
The first step is to obtain some algebraic decay at infinity of the solutions of \eqref{TW-LL}. This  can be achieved
following an argument of \cite{orlandi}.
\begin{prop}\label{decay-energy}
Assume that $c\in (0,1)$. Let  $u\in \E(\R^N)$ be a solution of \eqref{TW-LL}.
Suppose further that $u\in UC(\R^N)$ if $N\geq 3$.
 Then there exist constants $R_1,\alpha>0$ such that for all $R\geq R_1$,
\begin{equation}\label{decay-E}
\int_{B(0,R)^c}e(u)\leq  \left(\frac{R_1}R\right)^{\alpha}\int_{B(0,R_1)^c}e(u).
\end{equation}
\end{prop}

\begin{proof}
By Corollary~\ref{exis-lifting},  there exists $R_0>0$ such that equations \eqref{polar-1}--\eqref{polar-3}
hold on $B(0,R_0)^c$. Let $\rho> r\geq R_0$  and
 \begin{equation*}\Omega_{r,\rho}=\{ r\leq \abs{x} \leq \rho \}.
   \end{equation*}
 Multiplying \eqref{polar-1} by $\theta-\theta_r$, with $\theta_r=\frac1{\abs{\ptl B_r}}\int_{\ptl B_r}\theta$,
and integrating by parts, we get
\begin{equation}\label{dem-decay1}
\int_{\Omega_{r,\rho}}\varrho^2\grad{\theta}^2=c\int_{\Omega_{r,\rho}}u_3\ptl_1\theta+\int_{\ptl \Omega_{r,\rho}}(\theta-\theta_r)\varrho^2 \ptl_\nu\theta
-c\int_{\ptl\Omega_{r,\rho}}(\theta-\theta_r)u_3\nu_1,
\end{equation}
where $\nu$ denotes the outward normal to $\Omega_{r,\rho}$.

We recall that the Poincar\'e inequality for $\ptl B_r$ reads
\begin{equation*}
\int_{\ptl B_r}(\theta-\theta_r)^2\leq r^2\int_{\ptl B_r}\abs{\grad_\tau \theta}^2.
\end{equation*}
Then we obtain
\begin{align*}
  \left|\int_{\ptl B_{r}}(\theta-\theta_r)\varrho^2 \ptl_\nu\theta\right|\leq r \left(\int_{\ptl B_r}\abs{\grad \theta}^2\right)^{1/2}
\left(\int_{\ptl B_r}\abs{\rho \grad \theta}^2 \right)^{1/2}
\leq \frac{r}{\sqrt{1-\delta^2}} \int_{\ptl B_{r}} \abs{\rho \grad \theta}^2,
 \end{align*}
where $\delta=\norm{u_3}_{L^\infty(B_{r}^c)}$. Similarly, using also the inequality $ab\leq a^2/2+b^2/2$,
\begin{equation*}
  \left|\int_{\ptl\Omega_{r,\rho}}(\theta-\theta_r)u_3\nu_1\right|\leq \frac{r}{\sqrt{1-\delta^2}} \int_{\ptl B_r}e(u)\qquad  \textup{and} \qquad
\left|\int_{B_r^c}u_3\ptl_1\theta\right| \leq \frac{1}{\sqrt{1-\delta^2}} \int_{B_r^c}e(u).
\end{equation*}
On the other hand, by Lemma \ref{integ-Lp} and Corollary \ref{lema-theta-infty},
 \begin{equation*}(\theta-\theta_r)\varrho^2 \ptl_\nu\theta, (\theta-\theta_r)u_3\nu_1 \in L^2(B(0,R_0)^c).
   \end{equation*}
Then by Lemma \ref{est-ipp}, we conclude that there exists a sequence $\rho_n\to \infty$ such that
\begin{equation}\label{dem-decay2}
\int_{\ptl B_{\rho_n}}(\theta-\theta_r)\varrho^2 \ptl_\nu\theta\to 0\quad\textup{and}\quad \int_{\ptl B_{\rho_n}}(\theta-\theta_r)u_3\nu_1\to 0.
\end{equation}
Therefore, taking $\rho=\rho_n$, using \eqref{dem-decay1}--\eqref{dem-decay2} and the dominated convergence theorem we conclude that
\begin{equation*}
\int_{B_r^c}\varrho^2\grad{\theta}^2\leq \frac{c}{\sqrt{1-\delta^2}} \int_{ B_r^c}e(u)+\frac{3r}{\sqrt{1-\delta^2}} \int_{\ptl B_r}e(u).
\end{equation*}
In the same way, multiplying \eqref{polar-3} by
$u_3$, integrating by parts on the set $\Omega_{r,\tilde \rho_n}$, for a suitable sequence $\tilde \rho_n\to\infty$, we are led to
\begin{equation*}
\int_{B_r^c}(\abs{\grad u_3}^2+u_3^2)\leq (2\delta^2+c)\int_{ B_r^c}e(u)+ \int_{\ptl B_r}e(u).
\end{equation*}
Since $c<1$, we can choose $r$ large enough such that
\begin{equation*}
\frac{1}{2(1-\delta^2)}\left(2\delta^2+c\left( 1+\frac{1}{\sqrt{1-\delta^2}}\right)\right)<1.
\end{equation*}
Therefore, noticing that
\begin{equation*}
e(u)\leq \frac{1}{2(1-\delta^2)}(|\nabla u_3|^2 + \vr^2 |\nabla \theta|^2 + u_3^2),
\end{equation*}
we conclude that there exists a constant $K(\delta,c)>0$ such that
\begin{equation}\label{decay-eq}
\int_{B_r^c}e(u)\leq K(\delta,c)r \int_{\ptl B_r}e(u).
\end{equation}
Since  \begin{equation*}\frac{d}{dr} \int_{B_r^c}e(u)=-\int_{\ptl B_r}e(u),
               \end{equation*}
we can integrate inequality \eqref{decay-eq} to conclude that
 \begin{equation*}\int_{B^c_R}e(u)\leq  \left(\frac{r}R\right)^{1/K(c,\delta)}\int_{B^c_r}e(u),\qquad \textup{ for all }R\geq r,
   \end{equation*}
which completes the proof.
\end{proof}

\begin{cor}\label{cor-decay} Under the hypotheses and notations of Proposition~\ref{decay-energy}, we have
\begin{equation*}
\abs{\cdot}^\beta e(u)\in L^1(\R^N) \quad \textup{ and } \quad  \abs{\cdot}^\beta (\abs{F} +\abs{G_1}+\dots+\abs{G_N})\in L^1(\R^N),
\end{equation*}
for all $\beta \in [0,\alpha)$.
\end{cor}
\begin{proof}
Since $u\in C^\infty(\R^N)$, the fact that $\abs{\cdot}^\beta e(u)\in L^1(\R^N)$ is a direct
consequence of Proposition~\ref{decay-energy} (see e.g. \cite[Proposition 28]{gravejat-decay}).
On the other hand,  we take $R$ large enough such that $\norm{u_3}_{L^\infty(B_R^c)}\leq1/2$.
Then using that $\abs{u_3}\leq 1$, \eqref{bon-G} and \eqref{estim:polar}, we deduce that for all $j\in\{1,\dots, N\},$
\begin{equation*}
\abs{F}+\abs{G_j} \leq 2e(u)+\abs{u_3^2\ptl_1\theta}+\abs{u_3^2\ptl_j\theta}\leq \frac{\abs{\grad \theta}^2}{2}\leq \left(2+\frac{4}{\sqrt3}\right)e(u),
 \quad
\end{equation*}
and then the conclusion follows.
\end{proof}

The properties of the kernels appearing in equations \eqref{conv-u3} and \eqref{conv-theta} has been extensively studied in \cite{gravejat-decay}. Indeed, using the sets
\begin{align*}
 \MM_{k}(\R^N)&=\Big\{f:\R^N\to\C : \sup_{ x\in \R^N}\abs{x}^k\abs{f(x)} <\infty\Big\},\qquad k\in \N,\\
  \MM(\R^N)&=\Big\{f\in C^\infty(\R^N\setminus\{0\};\C) : D^k f \in \MM_{k}(\R^N)\cap \MM_{k+2}(\R^N), \textup{ for all }k\in \N \Big\},
\end{align*}
it is proved that
\begin{equation}\label{kernel-M}
  D^n \boL_c, D^n \boL_{c,j}, D^n \boT_{c,j,k} \in \MM_{\alpha+n}(\R^N), \textup{ for all }1\leq j,k\leq N,  \ n\in \N, \ \alpha\in (N-2,N],
\end{equation}
and also that
\begin{equation}\label{kernel-M2}
\wh \boL_c,  \wh \boL_{c,j}, \wh \boT_{c,j,k}\in \MM(\R^N).
\end{equation}
Similar results hold for the composed Riesz kernels $\mathcal R_{j,k}$. By combining these results with Corollary~\ref{cor-decay},
 equations \eqref{conv-u3} and \eqref{conv-theta} allow us to obtain the following algebraic decay.
\begin{lemma}\label{lemma-decay1} For any $n\in \N$,
 \begin{equation*}u_3,D^n( \grad(\chi \theta)), D^n(\grad  \check u)\in \MM_{N}(\R^N)\quad\textup{and}\quad D^nu_3\in \MM_{N+1}(\R^N).
   \end{equation*}
\end{lemma}
\begin{proof}
In view of Corollary~\ref{cor-decay}, the proof follows using the same arguments in \cite[Theorem~11]{gravejat-decay}.
\end{proof}

\begin{prop}\label{decay-infty}
Let $N\geq 2$ and $c\in (0,1)$. Assume that  $u\in \E(\R^N)$ is a  solution of \eqref{TW-LL}.
Suppose further that  $u\in UC(\R^N)$ if $N\geq 3$. Then there exist constants $R(u),K(c,u)\geq 0$ such that
\begin{align}
\label{dec-1}\abs{u_3(x)}+\abs{\grad\theta(x)}+\abs{\grad\check u(x)}&\leq \frac{K(c,u)}{1+\abs{x}^N},\\
 \label{dec-2}\abs{\grad u_3(x)}+\abs{D^2\theta(x)}+\abs{D^2\check u(x)}&\leq \frac{K(c,u)}{1+\abs{x}^{N+1}},\\
 \label{dec-3}\abs{D^2 u_3 (x)}&\leq \frac{K(c,u)}{1+\abs{x}^{N+2}},
\end{align}
for all $x\in B(0,R(u))^c$.
\end{prop}
\begin{proof}
Inequality \eqref{dec-1} and the estimate for $\grad u_3$ in \eqref{dec-2} are particular cases of Lemma~\ref{lemma-decay1}. A sightly improvement of Lemma~\ref{lemma-decay1}
is necessary for the decay of the second derivatives in \eqref{dec-2} and \eqref{dec-3}.
This can be done by following the lines in  \cite[Theorem 6]{gravejat-asymp}, which completes the proof.
\end{proof}

The pointwise convergence at infinity follows from general arguments in \cite{gravejat-asymp},
valid for all functions satisfying \eqref{kernel-M2}.
\begin{lema}[\cite{gravejat-asymp}]\label{lema-inf-1}
Assume that $T$ is a tempered distribution whose Fourier transform $\wh T=P/Q$ is a rational fraction which
belongs to $\MM(\R^N)$ and such that $Q\neq 0$ on $\R^N\setminus\{0\}.$ Then
there exists a function $T_\infty \in L^{\infty}(\S^{N-1};\C)$
such that
 \begin{equation*}R^{N} T(R\sigma)\to T_{\infty}(\sigma),\quad \textup{as }R\to\infty,\ \textup{ for all }\sigma\in \S^{N-1}.
   \end{equation*}
Moreover, assume that $f\in C^{\infty}(\R^N)\cap L^\infty(\R^N)\cap \MM_{2N}(\R^N)$.
Then $g\equiv  T*f$ satisfies
 \begin{equation*}R^Ng(R\sigma)\to T_{\infty}(\sigma) \intRR f(x)\,dx,\quad \textup{as }R\to\infty,\ \textup{ for all }\sigma\in \S^{N-1}.
   \end{equation*}
\end{lema}
Roughly speaking, it only remains to pass to the limit in the terms associated to the Riesz kernels
$\mathcal R_{i,j}$. For this purpose, we also recall the following.
\begin{lemma}[\cite{gravejat-asymp}]\label{lema-inf-2}
Assume that $f\in C^{\infty}(\R^N)\cap L^\infty(\R^N)\cap \MM_{2N}(\R^N)$ with $\grad f\in L^\infty(\R^N)\cap \MM_{2N+1}(\R^N)$.
Then $g\equiv \mathcal R_{j,k}*f$ satisfies for all $j,k\in\{1,\dots,N\}$,
 \begin{equation*}
  R^Ng(R\sigma)\to (2\pi)^{-\frac N2}\Gamma\left(\frac{N}{2}\right)(\delta_{j,k}-N\sigma_j\sigma_k)\intRR f(x)\,dx,\quad \textup{as }R\to\infty,\ \textup{ for all }\sigma\in \S^{N-1}.
   \end{equation*}
\end{lemma}

Finally, we have all the elements to provide the sketch of the proof of Theorem~\ref{limite-infty}.
\begin{proof}[Proof of Theorem~\ref{limite-infty}]
In view of \eqref{conv-u3}, \eqref{kernel-M2} and Lemma~\ref{lemma-decay1}, we can apply Lemma~\ref{lema-inf-1}
to the function $u_3$ to conclude that there exists $u_{3,\infty}\in L^{\infty}(\S^{N-1};\R)$
such that
\begin{equation}\label{conv-u-3}
R^N u_3(R\sigma)\to u_{3,\infty}(\sigma),\qquad \textup{as }R\to \infty, \ \textup{for all } \sigma\in \S^{N-1},
\end{equation}
where
\begin{align}\label{dem-u3}
 u_{3,\infty}(\sigma)=&
 \boL_{c,\infty}(\sigma)\intRR F-  c\sum_{j=1}^N\boL_{c,j,\infty}(\sigma)\intRR G_j,
 \end{align}
for some functions $\boL_{c,\infty}, \boL_{c,j,\infty}$.
Moreover, adapting \cite[Proposition~2]{gravejat-first}, we obtain
\begin{equation}\label{L}
\begin{split}
 \boL_{c,\infty}(\sigma)&=\frac{\Gamma\left(\frac{N}2\right)(1-c^2)^{\frac{N-3}{2}}c^2}{2\pi^\frac{N}2(1-c^2+c^2\sigma_1^2)^\frac{N}2}\left(1-\frac{N\sigma_1^2}{1-c^2+c^2\sigma_1^2}\right),\\
\boL_{c,j,\infty}(\sigma)&=\frac{\Gamma\left(\frac{N}2\right)(1-c^2)^{\frac{N-1}{2}}}{2\pi^\frac N2(1-c^2+c^2\sigma_1^2)^{\frac{N}2}}\left(\delta_{j,1}(1-c^2)^{-\frac{\delta_{j,1}+1}{2}}-\frac{N(1-c^2)^{-\delta_{j,1}}\sigma_1\sigma_j}{1-c^2+c^2\sigma_1^2}\right),
\end{split}
\end{equation}
which gives \eqref{id-inf2}.

Now we turn to equation \eqref{conv-theta}. Proceeding as before and  using also
Lemma~\ref{lema-inf-2}, we infer that there exist functions $\theta^j_\infty \in L^\infty(\S^{N-1};\R)$, $j\in\{1,\dots,N\}$
such that
\begin{equation}
R^{N}\ptl_j \theta (R\sigma)\to \theta^j_\infty(\sigma),  \qquad \textup{as }R\to \infty,
\end{equation}
for all  $j\in\{1,\dots,N\}$, and also that $\theta^j_\infty$ is given by
\begin{equation}\label{theta-j}
  \theta^j_\infty(\sigma)=c\boL_{c,j,\infty}(\sigma)\intRR F
-\sum_{k=1}^N\left(
c^2\boT_{c,j,k,\infty}(\sigma)+
\frac{\Gamma\left(\frac N2\right)}{2\pi^\frac N2}(\delta_{j,k}-N\sigma_j \sigma_k)
\right)\intRR G_k.
\end{equation}
As before, adapting \cite[Proposition~2]{gravejat-first} we  have
\begin{equation}
\label{T}
\begin{split}
  \boT_{c,j,k,\infty}=\frac{\Gamma\left(\frac{N}2\right)}{2\pi^\frac{N}{2} c^2}
\Bigg( (1-c^2)^\frac{N}2\bigg(
\frac{\delta_{j,k}(1-c^2)^{-\frac{\delta_{j,1}+\delta_{k,1}+1}{2}}}{(1-c^2+c^2\sigma_1^2)^\frac{N}2}-
\frac{N(1-c^2)^{-\delta_{j,1}-\delta_{k,1}+\frac12}\sigma_j\sigma_k}{(1-c^2+c^2\sigma_1^2)^\frac{N+2}{2}}\bigg)\\
 -\delta_{j,k}+N\sigma_j \sigma_k \Bigg).
\end{split}
\end{equation}
At this stage, we invoke Corollary~\ref{lema-theta-infty} and suppose that $\bar \theta=0$.
Then by \cite[Lemma~10]{gravejat-asymp},
\begin{equation}\label{conv-theta-1}
R\theta(R \sigma)\to \theta_\infty(\sigma)\equiv-\frac{1}{N-1}\sum_{j=1}^N \sigma_j\theta^j_\infty, \qquad \textup{as }R\to \infty.
\end{equation}
A further analysis shows that the convergence in \eqref{conv-u-3} and \eqref{conv-theta-1} are uniform, which implies
that
 \begin{equation*}
  R^{N-1}(\check u(R\sigma)-1)=R^{N-1}\left(\sqrt{1-u_3^2(R\sigma)}\exp(i\theta(R\sigma))-1\right)\to i\theta_\infty(\sigma),\quad \textup{in } L^{\infty}(\S^{N-1}).
   \end{equation*}
By combining with the expression for $\theta_\infty^j$ above, \eqref{lim-u-check} follows with $\lambda_\infty=1$ and $\check u_\infty=\theta_\infty$, provided that $\bar \theta=0$. Moreover, using \eqref{dem-u3}--\eqref{conv-theta-1} and that
\begin{align*}
&\sum_{j=1}^N \sigma_j \boL_{c,j, \infty}(\sigma)=-\frac{\Gamma\left(\frac N2\right)(N-1)(1-c^2)^{\frac{N-3}{2}}\sigma_1}{2\pi^\frac{N}2(1-c^2+c^2\sigma_1^2)^\frac N2},\\
&\sum_{j=1}^N \sigma_j \mathcal T_{c,j,k, \infty}(\sigma)=-\frac{\Gamma\left(\frac N2\right)(N-1)\sigma_k}{2\pi^\frac{N}2c^2}
 \left(\frac{(1-c^2)^{\frac{N}{2}-\frac12-{\delta_{k,1}}}}{(1-c^2+c^2\sigma_1^2)^\frac N2}-1\right),
\end{align*}
we obtain \eqref{id-inf1}.

In the case that $\bar \theta \neq 0$, it is enough to redefine the function $G$ in \eqref{G-LL}
as
 \begin{equation*}G=u_1\grad u_2-u_2\grad u_1-\grad(\chi (\theta-\bar \theta)),
   \end{equation*}
since then we can establish an equation such as \eqref{conv-theta} for
$\ptl_j(\chi (\theta-\bar \theta))$. Since $\theta(x)-\bar \theta\to 0$, as $x\to \infty,$
we conclude as before that there exists $\theta_\infty\in L^{\infty}(\S^{N-1};\R)$ such that
 \begin{equation*}R^{N-1}\left(\sqrt{1-u_3^2(R\sigma)}\exp(i(\theta(R\sigma)-\bar\theta)-1\right)\to i\theta_\infty(\sigma),\qquad \textup{in } L^{\infty}(\S^{N-1}).
   \end{equation*}
Since  $\sqrt{1-u_3^2(R\sigma)}\exp(i(\theta(R\sigma)-\bar\theta)=\check u(R\sigma)\exp(-i\bar \theta)$, taking $\lambda_\infty=\exp(i\bar \theta)$,
we conclude that
\begin{equation*}
R^{N-1}(\check u(R\sigma)-\lambda_\infty)\to i\lambda_\infty \theta_\infty, \qquad \textup{in } L^{\infty}(\S^{N-1}),
\end{equation*}
which completes the proof of Theorem~\ref{limite-infty}.
\end{proof}
\end{section}

\begin{merci}
  The author is grateful to N. Papanicolaou  and S.~Komineas for interesting and helpful discussions.
\end{merci}

\appendix
\setcounter{section}{1}
\def\thesection{\Alph{section}}
\section*{Appendix}
 \setcounter{equation}{0}
\setcounter{teo}{0}
\renewcommand\theteo{\thesection.\arabic{teo}}

For the convenience of the reader we recall some well-known results used in this paper.
We assume  $\Omega$ to be a smooth open bounded domain of $\R^N$.
 \begin{teo}[{\cite{stampacchia,koskela}}] \label{harmonic} Let $u\in H^1(\Omega)$, such that  $\Delta u=0$ on $D'(\Omega)$.
 Then there are constants $0<\alpha\leq 1$, $\alpha=\alpha(N)$, and $K>0$ such that if $x\in \Omega$ and
$0<\rho<r<\text{dist}(x,\Omega)$,
  \begin{equation*}\osc_{B_\rho}u \leq K\left(\frac{\rho}r\right)^\alpha \frac{\norm{u}_{L^{2}(B_r)}}{r^{N/2}}.
     \end{equation*}
Moreover, if $N=2$, then
 \begin{equation*}\osc_{B_\rho} u\leq K(\ln(\rho/r))^{-1/2} \norm{\grad u}_{L^{2}(B_r)},
   \end{equation*}
for some $K>0$.
 \end{teo}

\begin{teo}[{\cite{stampacchia}}] \label{stam-osc}
Let $p>N/2$ and $f\in L^p(\Omega)$. Assume that $u\in  H_0^1(\Omega)$ is solution of
 $-\Delta u=f$,  in $\Omega$.   Then $u$ is H\"older continuous in $\bar \Omega$.  Moreover, for $\rho>0$, there exists a constant
$K(\rho)$  such that
\begin{equation*}
\osc_{B_\rho \cap \Omega}u\leq K(\rho)\norm{f}_{L^p(\Omega)}.
\end{equation*}
 \end{teo}
\begin{lema}\label{decomp-L1}
 Let $f\in L^1(\R^N)$. Then for every $\ve>0$ there exists a constant  $K(\ve)$ such that
$f=f_1+f_2$ a.e. on $\R^N$ and
 \begin{equation*}\norm{f_2}_{L^1(\R^N)}\leq \ve,\qquad \norm{f_1}_{L^\infty(\R^N)}\leq K(\ve).
   \end{equation*}
\end{lema}
\begin{proof}
Let  \begin{equation*}f_{1,k}=
\begin{cases}
k, & \text{ if }f\geq k,\\
f, &  \text{ if }\abs{f}\leq k,\\
-k, &  \text{ if }f\leq -k.
\end{cases}
 \end{equation*}
and $f_{2,k}=f-f_{1,k}$. Then
\begin{equation}\label{dem-lema-stam}
\norm{f_{2,k}}_{L^1(\R^N)}\leq 2\int_{\{\abs{f}\geq k\}}\abs{f}.
\end{equation}
Since
 \begin{equation*}\abs{\{\abs{f}\geq k\}}=\int_{\{\abs{f}\geq k\}} 1\leq  \frac1k \norm{f}_{L^1(\R^N)}\to 0,\qquad \textup{as } k\to \infty,
   \end{equation*}
invoking the dominated convergence theorem and \eqref{dem-lema-stam}, we conclude that $\norm{f_{2,k}}_{L^1(\R^N)}\to 0$, as $k\to \infty$ and the conclusion follows.
\end{proof}

\begin{lema}\label{est-ipp}
Let $N\geq 1$. Assume that $f\in L^p(B(0,R_0)^c)$, for some $R_0\geq 0$ and $p\in [1,\infty)$.
Then there exists a sequence $R_n\to \infty$ such that
for all $s\in [0,N/p-N+1]$ we have
\begin{equation*}
R_n^s \int_{\ptl B(0,R_n)}\abs{f} d\sigma \leq \frac{ K(p,N)}{(\ln R_n)^{p}}, \qquad \textup{as }n\to \infty,
\end{equation*}
for some constant $K(p,N)>0$.
\end{lema}
\begin{proof}
 Since $f\in L^p(B(0,R_0)^c)$,
 \begin{equation*}\int_{R_0}^\infty \left(\int_{\ptl B(0,r)}\abs{f}^p\right)\,dr<\infty,
   \end{equation*}
and thus there is a sequence $R_n\to \infty$, as $n\to\infty$, such that
\begin{equation*}
 \int_{\ptl B(0,R_n)}\abs{f}^p\leq \frac1{R_n\ln(R_n)}.
\end{equation*}
Then, using the H\"older inequality we obtain
\begin{equation*}
\int_{\ptl B(0,R_n)}\abs{f}\leq ( K(N) R_n^{N-1})^{1-1/p}\frac{1}{(R_n\ln R_n)^{1/p}},
\end{equation*}
from where the result follows.
\end{proof}

\begin{lema}\label{grad-lp}
Let $c\geq 0$ and $u\in C^\infty(\R^N)\cap UC(\R^N)$ be a solution of \eqref{TW-LL}.
Assume that
 \begin{equation}\label{app-osc}
\osc_{B(y,r)}u\leq \frac{1}{8(1+c)(2s+1)},
\end{equation}
for some $y\in \R^N$, $r>0$ and $s\geq 1$. Then
\begin{equation}\label{app-grad}
\int_{B(y,r/2)}\abs{\grad u}^{2(s+1)}\leq 4(1+c)^2\left(1+\frac{16}{r^2}\right)\int_{B(y,r)}\abs{\grad u}^{2s}.
\end{equation}
\end{lema}
\begin{proof}
The ideas of the proof are based on classical computations for elliptic
equations with quadratic growth (see e.g. \cite{lady,borchers,jost}). Therefore
we only provide the main ideas, in order to show
 the dependence on $u,c,s$ and $N$ as stated.
We set $B_r\equiv B(y,r)$ and $\eta\in C_0^\infty(B_r)$ a function such that
$0\leq \eta\leq 1$,
\begin{equation}\label{cota-eta}
 \abs{\grad \eta}\leq\frac{4}{r} \textup{ on }B_r  \qquad {\rm and} \qquad\eta\equiv 1 \textup{ on } B_{r/2} .
\end{equation}
Finally, we fix $w=\abs{\grad u}^2$, which is smooth by hypothesis, so that
\begin{equation}\label{der-w-L}
\abs{\grad w}\leq 2w^{1/2}\abs{D^2u}.
\end{equation}
We now divide the computations in several steps.
\setcounter{step}{0}
\begin{step}\label{step1}
If $\osc_{B_r}{u}\leq 1/4$, we have
\begin{equation*}
\int_{B_r}\eta^2 w^{s+1}\leq 2\osc_{B_r}{u}\left(\int_{B_r}\abs{\grad \eta}^2 w^s+\frac{2s+1}{2}\int_{B_r}\eta^2\abs{D^2u}w^{s-1}\right).
\end{equation*}
\end{step}
\noindent Indeed, since
\begin{equation*}
\int_{B_r}\eta^2 w^{s+1}=\int_{B_r}\eta^2 w^s \grad{(u-u(y))}\cdot\grad{u},
\end{equation*}
integrating by parts and using \eqref{der-w-L}, we deduce that
\begin{equation}\label{norm-w}
\int_{B_r}\eta^2 w^{s+1}\leq \osc_{B_r}{u}
\left(2\int_{B_r}\eta \abs{\grad \eta} w^{s+1/2}+(2s+1)\int_{B_r}\eta^2\abs{D^2u}w^{s}\right).
\end{equation}
Using the elementary inequalities
$2ab\leq a^2+b^2$ and $ab\leq a^2+b^2/4$
 in the first and second
integrals in the r.h.s. of \eqref{norm-w}, we obtain
\begin{equation*}
  (1-2\osc_{B_r}{u}) \int_{B_r}\eta^2 w^{s+1}\leq \osc_{B_r}{u} \left(\int_{B_r}\abs{\grad \eta}^2 w^s+\frac{(2s+1)^2}4
\int_{B_r}\eta^2\abs{D^2u}w^{s-1}\right).
\end{equation*}
Since $\osc_{B_r}{u}\leq 1/4$, we conclude Step~\ref{step1}.
\begin{step}\label{step3}We have
\begin{equation*} \frac12 \int_{B_r}\eta^2\abs{D^2 u}^ 2w^{s-1}\leq {2} \int_{B_r} \abs{\grad \eta}^2 w^s -
\sum_{k=1}^N\int_{B_r} \ptl_k(\Delta u)\cdot \ptl_k u\, \eta^2 w^{s-1}.
\end{equation*}
\end{step}
\noindent Let $k\in\{1,\dots,N\}$ and $\phi_k=\eta^2w^{s-1}\ptl_k u\in C_0^\infty(B_r)$. On one hand, integrating by parts,
\begin{equation}\label{gradd-1}
\sum_{j=1}^N \int_{B_r} \ptl^2_{jk} u\cdot \ptl_j \phi_k=-\int_{B_r} \ptl_k (\Delta u)\cdot \phi_k.
\end{equation}
On the other hand, using that  \begin{equation*}\ptl_j w =2\sum_{k=1}^N \ptl_k u\cdot \ptl^2_{jk}u,
\end{equation*}
and developing the term $\ptl_j \phi_k$,
\begin{equation}\label{app-ipp}
\begin{split}
\sum_{j,k=1}^N \int_{B_r}\ptl^2_{jk} u\cdot \ptl_j \phi_k
= &\int_{B_r} \eta^2\abs{D^2 u}^2 w^{s-1}
+2\sum_{j,k=1}^N \int_{B_r} \eta \ptl_j \eta w^{s-1}\ptl_k u \cdot \ptl^2_{jk}u \\
&+\frac{s-1}2\int_{B_r} \eta^2w^{s-2}\abs{\grad w}^2.
\end{split}
\end{equation}
Then the conclusion of this step follows combining \eqref{gradd-1} and \eqref{app-ipp}, noticing that
the last integral in the r.h.s. of \eqref{app-ipp} is nonnegative, and
that
\begin{equation*}
  2\sum_{j,k=1}^N \abs{\eta \ptl_j \eta w^{s-1}\ptl^2_{jk}u \cdot \ptl_k u}\leq 2\sum_{j,k=1}^N \eta\abs{\ptl^2_{jk}u}w^{\frac{s-1}{2}} \cdot \abs{\ptl_j \eta}\abs{\ptl_k u}w^{\frac{s-1}2}
\leq  \frac12\eta^2\abs{D^2 u}^2 w^{s-1}+2\abs{\grad \eta}^2w^s.
\end{equation*}
\begin{step}\label{step2} For all $\delta>0$, we have
\begin{equation*}
\sum_{j=1}^N \abs{\ptl_j{\Delta u} \cdot \ptl_j u}\leq \delta(c+1)\abs{D^2u}^2+\left(c+\frac{c}{\delta}+4\right)w+ (1+c)w^2.
\end{equation*}
\end{step}
\noindent Using \eqref{TW-LL} and the fact that $\abs{u}=1$, it is simple to check  that
\begin{equation*}
\sum_{j=1}^N \abs{\ptl_j{\Delta u} \cdot \ptl_j u}\leq 2w\abs{D^2u}+w^2+4w+2cw^{3/2}+2c\abs{D^2u}w^{1/2}.
\end{equation*}
By combining with the fact that $2ab\leq \delta a^2+\delta^{-1}b^2$, for all $\delta>0$, we finish Step~\ref{step2}.
\begin{step}\label{step4}\hspace{-1cm}.
 \begin{equation*}
   \frac14 \int_{B_r}\eta^2\abs{D^2u}^2 w^{s-1}\leq  2\int_{B_r}\abs{\grad \eta}^2w^s+
 (4c^2+5c+4)\int_{B_r}\eta^2 w^s+
 (c+1)\int_{B_r}\eta^2 w^{s+1}.
 \end{equation*}
\end{step}
\noindent Step \ref{step4} follows immediately from Steps~\ref{step3} and \ref{step2}, taking $\delta=(4(c+1))^{-1}$.

Now we are  in position to finish the proof of Lemma~\ref{grad-lp}. In fact, by combining  Steps~\ref{step1} and \ref{step4}, we are led to
\begin{equation*}
\begin{split}
  (1-4(c+1)(2s+1)\osc_{B_r}u) \int_{B_r}\eta^2 w^{s+1}
\leq & \  2 \osc_{B_r}u\left(
(8s+5)\int_{B_r}\abs{\grad \eta}^2w^s+\right.\\
&\left.2(4c^2+5c+4)(2s+1)\int_{B_r} \eta^2w^s
\right).
\end{split}
\end{equation*}
Also, we see that  \eqref{app-osc} implies that
 \begin{equation*}1/2\leq (1-4(c+1)(2s+1)\osc_{B_r}u),
   \end{equation*}
 and that
\begin{equation*}
  2\osc_{B_r}u \cdot \max\{8s+5 ,2(4c^2+5c+4)(2s+1)\}\leq \frac{2(4c^2+5c+4)}{4(1+c)}\leq 2(1+c^2).
 \end{equation*}
By combining with \eqref{cota-eta}, we conclude \eqref{app-grad}.
\end{proof}

\bibliography{ref}
 \bibliographystyle{abbrv}
\end{document}